%% file: nakiboglu.tex
\documentclass[12pt,twoside,a4paper]{report}
\usepackage[scale=0.7, heightrounded]{geometry}

\usepackage[english]{babel}
\usepackage[latin1]{inputenc}		
\usepackage[T1]{fontenc}				
\usepackage{graphicx}						
\usepackage{lmodern}						
\usepackage{type1cm}               

\usepackage{amsmath}
\usepackage{amsthm}
\usepackage{amsfonts}
\usepackage{amssymb}
\usepackage{yhmath} 
\usepackage{mathrsfs}
\usepackage{enumerate}

\usepackage[all, knot]{xy}
   \xyoption{arc} 
\usepackage{textpos}  

\usepackage{url}

\newtheorem{thm}{Theorem}[section]

\makeatletter
\newcommand*{\c@thmprop}{\c@thm}
\newcommand*{\p@thmprop}{\p@thm}

\makeatother
\newtheorem{prop}[thmprop]{Proposition} 

\makeatletter
\newcommand*{\c@thmrem}{\c@thm}
\newcommand*{\p@thmrem}{\p@thm}

\makeatother
\newtheorem{rem}[thmrem]{Remark}

\makeatletter
\newcommand*{\c@thmdefi}{\c@thm}
\newcommand*{\p@thmdefi}{\p@thm}

\makeatother
\newtheorem{defi}[thmdefi]{Definition}

\makeatletter
\newcommand*{\c@thmlem}{\c@thm}
\newcommand*{\p@thmlem}{\p@thm}

\makeatother
\newtheorem{lem}[thmlem]{Lemma}

\newtheorem*{claim}{Claim} 
\makeatletter
\newcommand*{\c@thmnote}{\c@thm}
\newcommand*{\p@thmnote}{\p@thm}

\makeatother
\newtheorem{note}[thmnote]{Notation}

\makeatletter
\newcommand*{\c@thmdefiprop}{\c@thm}
\newcommand*{\p@thmdefiprop}{\p@thm}

\makeatother
\newtheorem{defiprop}[thmdefiprop]{Definition and Proposition}

\makeatletter
\newcommand*{\c@thmex}{\c@thm}
\newcommand*{\p@thmex}{\p@thm}

\makeatother
\newtheorem{example}[thmex]{Example}

\newtheorem*{defin*}{Definition} 
\newtheorem*{propo*}{Proposition} 
\newtheorem*{defipropo*}{Definition and Proposition} 
\newtheorem*{them*}{Theorem}

\newcommand{\sys}[1]{#1^{\mathbf{s}}}

\usepackage{setspace}

\allowdisplaybreaks

\include{metadata}							
\include{compatibility}

\addto\extrasenglish{%
} 
\addto\extrasenglish{%
} 
 
\addto\extrasenglish{%
} 

\addto\extrasenglish{%
} 

\usepackage[
LabelsAligned,					
NoDate									
]{currvita}

\usepackage{fancyhdr}

\begin{document}

\pagenumbering{roman}


\include{titel_goe}

\include{abstract}

\newpage~
\thispagestyle{empty}




\include{acknowledgement}
\newpage~
\thispagestyle{empty}

\pagenumbering{roman}
\setcounter{page}{6}
\tableofcontents
\thispagestyle{plain}



\pagestyle{fancy}					
\renewcommand{\chaptermark}[1]{\markboth{\MakeUppercase{\thechapter.\ #1}}{\MakeUppercase{\thechapter.\ #1}}}
\setlength{\headheight}{15.17pt}
\fancyfoot{}
\fancyhead[LE,RO]{\thepage}
\fancyhead[LO]{\bfseries  \leftmark}
\fancyhead[RE]{\bfseries  \rightmark}
\clearpage
\setcounter{chapter}{-1}
\pagenumbering{arabic}
\setcounter{page}{9}
\include{motivation}

\newpage~
\thispagestyle{empty}	

\renewcommand{\thechapter}{\Roman{chapter}}
\renewcommand{\thesection}{\thechapter.\arabic{section}}
\renewcommand{\thesubsection}{\thechapter.\arabic{section}.\arabic{subsection}}

\include{grundlagen} 

\newpage~
\thispagestyle{empty}	
\include{onehalfsquaring}

\include{oriented}

\include{Thom_compat}

\include{computations}
\newpage~
\thispagestyle{empty}

\include{outlook}




\bibliographystyle{alphaurl}
\bibliography{bibliography}










\end{document}

%% file: metadata.tex
\newcommand{\dcauthorsurname}{Krempasky (geb. Nakiboglu)} 
\newcommand{\dcauthorname}{Denise}

\newcommand{\dctitle}{Symmetric Squaring Construction in Bordism} 
  
\newcommand{\dckeydea}{Schlagwort 1}
\newcommand{\dckeydeb}{Schlagwort 2}
\newcommand{\dckeydec}{Schlagwort 3}
\newcommand{\dckeyded}{Schlagwort 4}

\newcommand{\dckeyena}{algebraic topology}
\newcommand{\dckeyenb}{homology theory}
\newcommand{\dckeyenc}{bordism}
\newcommand{\dckeyend}{\v{C}ech homology}

\newcommand{\dckeywordsen}{\vfill \raggedright {\textbf{Keywords:}}\\ \dckeyena, \dckeyenb, \dckeyenc, \dckeyend \\}

\newcommand{\dcpdfsubject}{Dissertation}

%% file: compatibility.tex
\usepackage{ifpdf}

\ifpdf
\usepackage[
	unicode,
	pdftitle={\dctitle},
	pdfauthor={\dcauthorsurname\ \dcauthorname},
	pdfsubject={\dcpdfsubject}, 
	pdfkeywords={\dckeydea, \dckeydeb, \dckeydec, \dckeyded},
	pdfpagemode=UseOutlines,
  colorlinks=true,					
	linkcolor=black,					
	filecolor=black,					
	urlcolor=black,						
	citecolor=black,					
	pdftex=true,              
	plainpages=false,         
	hypertexnames=false,      
	pdfpagelabels=true,       
	hyperindex=true]{hyperref}
\else
	
\fi

%% file: titel_goe.tex
\thispagestyle{empty}
\begin{center}
\vspace*{2.0cm}
\textsl{\Huge  Symmetric Squaring ~\\
\vspace{1ex}
        in Homology and Bordism~\\ 
        }
\vspace*{2.0cm}
{\large Dissertation ~\\[2ex]
        zur Erlangung des mathematisch-naturwissenschaftlichen Doktorgrades ~\\[1ex]
        {\bf ``Doctor rerum naturalium''} ~\\[1ex]
        der Georg-August-Universit\"at G\"ottingen \par}
~\vspace*{2.0cm}
{\large vorgelegt von ~\\[1ex]
        {\bf Denise Krempasky (geb. Nakibo\u{g}lu)} ~\\
        aus Coburg \par}
\vspace*{3.0cm}
{\large G\"ottingen 2011 \par}
\end{center}
\newpage~
\thispagestyle{empty}
\vspace*{10cm}
\vfill
~\\
{\bf Referenten der Dissertation:}~\\
Referent: Prof.\ Dr.\ Thomas Schick~\\
Korreferent: Prof.\ Dr.\ Ralf Meyer~\\
~\\
{\bf Tag der m\"undlichen Pr\"ufung:}~\\

%% file: abstract.tex


\begin{abstract}
\setcounter{page}{5} 
Looking at the cartesian product $X\times X$ of a topological space $X$ with itself, a natural map to be considered on that object is the involution that interchanges the coordinates, i.e. that maps $(x,y)$ to $(y,x)$.
The so-called 'symmetric squaring construction'  in \v{C}ech homology with $Z/2$-coefficients was introduced by Schick et al. 2007 as a map from the $k$-th \v{C}ech homology group of a space $X$ to the $2k$-th \v{C}ech homology group of $X \times X$ divided by the above mentioned involution. It turns out to be a crucial construction in the proof of a parametrised Borsuk-Ulam Theorem.\\
The symmetric squaring construction can be generalized to give a map in bordism, which will be the main topic of this thesis. More precisely, it will be shown that there is a well-defined, natural map from the $k$-th singular bordism group of $X$ to the $2k$-th bordism group of $X\times X$ divided by the involution as above. Moreover, this squaring really is a generalisation of the \v{C}ech homology case since it is compatible with the passage from bordism to homology via the fundamental class homomorphism. On the way to this result, the concept of \v{C}ech bordism is first defined as a combination of bordism and \v{C}ech homology and then compared to \v{C}ech homology.
\dckeywordsen				
\end{abstract}

%
%

%% file: acknowledgement.tex

\chapter*{Acknowledgement}
\thispagestyle{empty}
Many people influenced me and my work in a way such that this thesis had a chance to come into being. I am grateful to my supervisor Prof. Dr. Thomas Schick, who encouraged me and who always had the appropriate questions and answers needed to go on. Furthermore, I want to thank Prof. Dr. Ralf Meyer for taking the Koreferat.\\ 
For financial support during the writing of this thesis I thank the DFG Research Training Group 535 'Gruppen und Geometrie'.\\
Especially in the last few months it was of great importance to organize my work as flexible as I could. Thank you for making that possible: Prof. Dr. Stefan Halverscheid, Prof. Dr. Anita Schöbel, Prof. Dr. Ina Kersten and Dr. Hartje Kriete.\\
For very helpful math and non-math discussions I appreciate my colleagues Nils Waterstraat, Manuel Köhler, Ulrich Pennig and Alessandro Fermi.\\
For the making of pictures of cobordism, I used the tutorial and one of the templates provided by Aaron Lauda at \url{http://www.math.columbia.edu/~lauda/xy/cobordismtutorial/index.html}. I hereby want to thank him for making these available.\\

Last but not least, I am grateful to my whole family for being there, my Mum for making me want to understand the world and my brother for the unique connection we share.\\
I thank Thorsten for finding typos and more and Thorsten and Daria for their love, patience and support and for bringing me down to earth whenever I need it. Ihr seid der Jackpot meines Lebens.

%% file: motivation.tex
\chapter{Introduction}

There are many possibilities of producing new topological spaces from a given topological space $X$. Two of these form the basis of the construction that is the main objective of this thesis. Namely, these are building the cartesian product of $X$ with itself on the one hand and dividing by some relation on the space on the other hand. The symmetric squaring construction examined throughout this thesis is a composition of both of these possibilities. It consists of first squaring the space and then dividing the result by a symmetric relation. To be precise, in the cartesian square $X\times X$, the elements $(x,y)$ and $(y,x)$ are identified for all $x$ and $y$ in $X$ to give the symmetric square of $X$ which shall be denoted by $\sys{X}$ from now on.\\
In \cite{borsuk} the symmetric squaring construction was introduced and used in the context of \v{C}ech homology with $\mathbb Z_2$-coefficients. More precisely, there is constructed a well-defined map\[
\sys{(\hspace{0,5em}\cdot\hspace{0,5em})}\colon \check{H}_{k}(X;\mathbb Z_2)\rightarrow \check{H}_{2k}\left(\sys{X}, pr(\Delta);\mathbb Z_2 \right).\footnote{Here $\Delta$ denotes the diagonal in $X\times X$ and $pr\colon X\times X\rightarrow \sys{X}$ is the canonical projection.}
\] This map is needed for a proof of a generalised Borsuk-Ulam Theorem where it is crucial to be able to construct a specific non-trivial homology class. Symmetric squaring is a valuable tool for this purpose because it behaves very well with respect to fundamental classes of manifolds, which are non-trivial homology classes by definition. Namely, it has the property of mapping the fundamental class of a manifold to the \v{C}ech homology version of the fundamental class of the symmetric square of that manifold.\\ However, symmetric squaring as such is considered to be a construction of independent interest which is worth to be examined in other contexts. Our purpose in this thesis is to provide a generalisation of symmetric squaring to the context of bordism.\\
Homology and bordism have always been related. As pointed out in \cite{MR2058291}, one of the reasons Ren\'e Thom introduced the notion of bordism in \cite{MR0061823} in the first place was that he wanted to answer a question raised by Steenrod in \cite{MR0030189} related to homology classes.\footnote{The question is posed as problem 25 in \cite{MR0030189}: Given a homology class in some finite simplicial complex. Does there always exist a manifold and a map of that manifold into the given complex, such that the image of the fundamental class by that map in homology is the given homology class?} This question also led to the definition of a canonical map from bordism to homology called fundamental class homomorphism, which we recall in detail \autoref{fundhom} in \autoref{sec:compat}:
\begin{align*}
 \mu\colon\Omega_k(X,A)&\rightarrow H_k(X,A,\mathbb Z_2) \text{ is defined via }\\
\left[M,\partial M; f\right]&\mapsto \mu(M,\partial M,f):=(f)_\ast(\sigma_{\mathbf f}),
\end{align*}
where $\sigma_{\mathbf f}\in H_k(M,\partial M,\mathbb Z_2)$ is the fundamental class of $(M,\partial M)$.\\
Furthermore, what makes symmetric squaring useful in connection with homology is its behaviour with respect to manifolds and their fundamental homology classes, as pointed out above. Since the main objects bordism deals with are compact manifolds, we now already have encountered two indications why it is interesting to try to relate symmetric squaring to bordism.\\ 
On our way to a generalisation of symmetric squaring to bordism, we encounter some difficulties that lead to the definition of an alternative bordism functor in \autoref{cechbord} that combines bordism with \v{C}ech homology. It is called \v{C}ech bordism due to its analogy to \v{C}ech homology and its definition is given in \autoref{cechbordism}.
\begin{defipropo*}
Let $n\in \mathbb N$ be a natural number. For every topological pair of spaces $(X,A)$, we define the limit groups $\check{\mathcal{N}}_n(X,A)$ as 
\begin{align*}
\check{\mathcal{N}}_n(X,A)&:=\varprojlim\left\{\mathcal{N}_{n}(X,U)\,|\,  A \subset U \text{ is an open neighbourhood of } A \text{ in } X \right\},
\end{align*}where $\mathcal{N}_\ast$ denotes the unoriented singular bordism functor.\footnote{Defined for example in \cite{MR548463}.}
This gives a \v{C}ech-bordism-functor from the category of topological pairs of spaces to the category of groups.
\end{defipropo*}
Exactly as in the homology case, \v{C}ech bordism groups are isomorphic to singular bordism groups in many cases. In \autoref{prop:singcech} in \autoref{cechbord} we prove the following.\footnote{Note \autoref{rem:readability} about the usage of oriented and unoriented bordism in \autoref{cechbord}.}
\begin{propo*}
 Let $(X,A)$ be such that $X$ is an ENR\footnote{Cf. \autoref{def:ENR} on page \pageref{def:ENR}.} and $A\subset X$ is an ENR as well. Then \[\check{\mathcal{N}}(X,A)\simeq {\mathcal{N}}(X,A).\]
\end{propo*}
Having provided similar conditions to the homology case in the world of bordism, a generalisation of the symmetric squaring construction can be established in \autoref{halfun} and \autoref{unnat} in \autoref{uncase}, namely:
\begin{them*}
 Let $(X,A)$ be a pair of topological spaces. Then there exists a well-defined symmetric squaring map in unoriented bordism \[
\sys{(\hspace{0,5em}\cdot\hspace{0,5em})}\colon\mathcal{N}_n(X,A)\rightarrow\check{\mathcal{N}}_{2n}\left(\sys{(X,A)}\right).
\] 
This construction is natural.
\end{them*}
That this really is a generalisation of the result in homology becomes clear when combining it with the above noted passage from bordism to homology. Symmetric squaring is compatible with this fundamental class homomorphism and the main reason for this is the good behaviour of symmetric squaring with respect to fundamental classes. In \autoref{sec:compat}, we prove that the diagram
\[
\begin{xy}\xymatrixcolsep{3.5pc}
\xymatrix{
\mathcal{N}_n(X,A) \ar[r]^{\mu} \ar[d]_{\sys{\cdot}} & H_n(X,A,\mathbb Z_2)\ar[d]_{\sys{\cdot}} \\
\check{\mathcal{N}}_{2n}(\sys{(X,A)}) \ar[r]_{\check{\mu}} & \check{H}_{2n}(\sys{(X,A)},\mathbb Z_2)
}
\end{xy}
\] commutes.\\
Until now, all stated results used unoriented homology or $\mathbb Z_2$ coefficients respectively. But we can also show that similar results hold in oriented cases. In \cite{diplom} symmetric squaring is interpreted as a map in \v{C}ech homology with integer coefficients. The important point to note here, however, is that this generalisation does not hold for all dimensions, since symmetric squaring induces an orientation reversing map in odd dimensions. This suggests attempting a generalisation of symmetric squaring to oriented bordism at least in even dimensions.\\
What we can prove in \autoref{orcase} and \autoref{sec:compat} is that similar to the unoriented case symmetric squaring induces a well-defined map in oriented bordism in even dimensions and that this is compatible with the fundamental class transformation:
\begin{them*}
  Let $(X,A)$ be a pair of topological spaces and let $n\in\mathbb N$ be even. Then there exists a well-defined and natural symmetric squaring map \[
\sys{(\hspace{0,5em}\cdot\hspace{0,5em})}\;\;\colon\Omega_n(X,A)\rightarrow\check{\Omega}_{2n}\left(\sys{(X,A)}\right)
\]with the property that the diagram 
\[
\begin{xy}\xymatrixcolsep{3.5pc}
\xymatrix{
\Omega_n(X,A) \ar[r]^{\mu} \ar[d]_{\sys{\cdot}} & H_n(X,A,\mathbb Z)\ar[d]_{\sys{\cdot}} \\
\check{\Omega}_{2n}(\sys{(X,A)}) \ar[r]_{\check{\mu}} & \check{H}_{2n}(\sys{(X,A)},\mathbb Z)
}
\end{xy}
\]commutes.\footnote{The map $\mu$ is the so called fundamental-class homomorphism. It is defined together with its induced map $\check{\mu}$ in \autoref{sec:compat}.}
\end{them*}
It is natural to try to compute the homological or bordism symmetric squaring construction for some special spaces. However, computing homology groups of symmetric squared spaces is not easy in general since the involution $(x,y)\mapsto (y,x)$ that we divide by during the symmetric squaring construction has the diagonal of the squared space as a fixed point set. Therefore it does not induce a free action on the squared space. To give an idea of what could make computations in this setting easier, the theory of $G$-spaces is used in \autoref{sec:G-spaces} to prove that for compact and metrizable spaces $X$ there is an isomorphism
\[
\check{H}_i\left(\sys{X},\Delta\right)\approx \check{H}_i\left(\left(X\times X\right)_{\mathbb Z_2}, \Delta\times \mathbb{RP}^\infty \right)
\]
where $\Delta$ denotes the diagonal in $X\times X$. Here, the so-called Borel construction\footnote{See \cite{MR1236839} for details.} $\left(X\times X\right)_{\mathbb Z_2}$ is used which is defined as the orbit space of the diagonal action induced by the coordinate-swapping involution  on the product $(X\times X)\times S^\infty$.\footnote{We implicitly use the universal principal bundle $\mathbb Z_2\rightarrow S^\infty\rightarrow \mathbb{RP}^\infty$ here.}

This thesis contains three chapters. The first of these includes a review of former results concerning symmetric squaring as well as the introduction of the bordism theory named \v{C}ech bordism. In the second chapter we use the achievements of the first one to prove the existence of well-defined symmetric squaring maps in oriented and unoriented bordism and show that these are compatible with the homological symmetric squaring known before. We conclude in the third chapter by giving possibilities to compute symmetric squaring maps for specific topological spaces via the Borel construction.

%% file: grundlagen.tex
\chapter{Background on homology and bordism}

At the core of all considerations in this thesis lies the construction\[
X\leadsto X\times X \leadsto X\times X/\tau=\sys{X},
\]called the symmetric squaring construction. It can be applied to all topological spaces $X$ and is performed as follows. The first step of the construction is taking the cartesian product $X\times X$ of the space $X$. In the second step this product is divided by the coordinate-interchanging involution \begin{align*}\tau\colon X\times X&\rightarrow X\times X\\
(x,y)&\mapsto(y,x).
\end{align*} to give $\sys{X}:=X\times X/\tau$.\\
This chapter is split up into two sections. In the first section we will see how the symmetric squaring construction lives in the world of singular and \v{C}ech homology as well as we will take a short look at how it has been used in \cite{borsuk}.\\
The goal of generalising the construction to singular bordism leads us to the definition of \v{C}ech bordism. The second section of this chapter contains its definition and a comparison between \v{C}ech and singular bordism.

\section{Results in homology}\label{resultshom}
Symmetric squaring was defined as a homological construction in the context of \v{C}ech homology with $\mathbb Z_2$-coefficients in \cite{borsuk}, where it was also used to prove a parametrised Borsuk-Ulam Theorem. The property that makes this construction useful is that for smooth compact manifolds the symmetric squaring induces a map in homology which maps the fundamental class of a manifold to the fundamental class of the symmetric square of that manifold. A generalisation of the construction and the result just stated concerning fundamental classes in \v{C}ech homology with integer coefficients can be found in \cite{diplom}. This section shall give more details about the facts just summarised.

\subsection{Symmetric squaring in \v{C}ech-homology}

Remember that symmetric squaring is the operation of performing a squaring first and then dividing the result by a symmetric map. More precisely we define the symmetric square as follows.

\begin{defi}[Symmetric Squaring]\label{sys}Let $X$ be a topological space and define the coordinate-switching involution $\tau$ by
\begin{align*}
 \tau\colon X\times X&\rightarrow X\times X\\
(x,y)&\mapsto (y,x).
\end{align*}
Then the topological space $X\times X/\tau$ is called the symmetric square of $X$ and it will be denoted by $\sys{X}$. For pairs of topological spaces $(X,A)$, the symmetric squaring is defined in an analogous manner as
\begin{align*}
 \sys{(X,A)}:=(pr(X\times X),pr(X\times A\cup A\times X\cup \Delta)),
\end{align*}
where $pr\colon X\times X\rightarrow X\times X/\tau$ denotes the canonical projection and $\Delta$ denotes the diagonal $\left\{(x,x)|x\in X\right\}\subset X\times X$ in $X\times X$. In particular, $\sys{(X,\emptyset)}=(\sys{X},pr(\Delta))$.
\end{defi}

 The diagonal is added to the subspace of $\sys{(X,A)}$ for technical reasons. Especially if $X$ is a smooth manifold, it is necessary to cut out the diagonal or to at least work relative to it in homology. Since the involution $\tau$ leaves the diagonal fixed, the quotient by $\tau$ does not have a canonical manifold structure there. Outside the diagonal, however, there is a smooth structure which can and will be used to think of $\sys{X}\setminus pr(\Delta)$ as a smooth manifold whenever $X$ is a smooth manifold.\footnote{Details on this are given in \autoref{diagsec}.}\\
The homological symmetric squaring is first introduced on the level of singular chains, which is then shown to induce a well-defined mapping in \v{C}ech homology.\footnote{For homology with $\mathbb Z_2$-coefficients compare Section 3 of \cite{borsuk}. Integer coefficients are used in Chapter 2 of \cite{diplom}.} In detail, the definition on chain level is as follows.

\begin{defi}[Symmetric Squaring in Homology]\label{hom}
Let $k\in \mathbb Z$ be an even integer and let $(X,A)$ be a topological pair. For $\sigma \in H_k(X,A,\mathbb Z)$ choose a representation $\sigma=\sum\limits_{i=1}^n {g_i\sigma_{i}}$  by elements of the $k$-th singular chain group $C_k(X,A,\mathbb Z)$. Then the assignment
\begin{align*}
\sigma=\sum_{i=1}^n {g_i\sigma_{i}} &\mapsto \sys{\sigma} :=
 \sum\limits_{\genfrac{}{}{0pt}{}{i<j}{1\leq i,j \leq n}}{g_i g_j(pr)_{\sharp}(\sigma_i \times \sigma_j)},
\end{align*} induces a map $\sys{(\cdot)}\colon {H}_k(X,A,\mathbb Z)\rightarrow {H}_{2k}(\sys{(X,A)},\mathbb Z)$ in homology. This map shall be called symmetric squaring map in homology.\\
Here $\times$ denotes the simplicial cross product\footnote{An introduction to the simplicial cross product can be found in Section 3.B of \cite{MR1867354}.} and $(pr)_\sharp$ is the chain map induced by the projection $pr\colon X\times X\rightarrow X\times X/\tau.$
\end{defi}
There are two things worth noticing about this definition. The first one is the usage of even dimensions only represented by the usage of even integers $k\in \mathbb Z$. While dealing with chain groups with integer coefficients, we have to take the orientation of the simplices into account. It can be shown\footnote{See Lemma 2.3 in \cite{diplom}.} that the projection induced by the coordinate-interchanging map $\tau$ is an orientation reversing map on simplicial level in odd dimensions as well as it is an orientation preserving map in even dimensions. This limits further considerations to even dimensions when dealing with $\mathbb Z$-coefficients, since there is no canonical orientation of the simplices in the quotient by $\tau$ in odd dimensions.\\
The second aspect which is important to note here is that in the definition of the chain $\sys{\sigma}$ only half of the product simplices $\sigma_i\times\sigma_j$ are added. This alludes to the fact that symmetric squaring can as well be thought of as constructing 'half of the square'.\footnote{Compare Section 1 in \cite{borsuk}.} Giving a detailed explanation of this property relies on the usage of a special way of defining \v{C}ech homology in this context.\footnote{This definition is also used in Section VIII.13 of \cite{MR1335915}.} More precisely, consider the neighbourhoods $U\subset X$ of the subspace $A\subset X$ in a topological pair $(X,A)$ as a quasi-ordered set ordered by inverse inclusion. Then the \v{C}ech homology group of the pair $(X,A)$ is defined to be the inverse limit of the singular homology groups of $(X,U)$ over this quasi-ordered set. This possibility of defining \v{C}ech homology of topological pairs $(X,A)$ as an inverse limit over singular homology groups of neighbourhoods is one of the main reasons why \v{C}ech homology is used throughout this thesis.
\begin{defi}[\v{C}ech Homology]\label{cechhom}
Let $k\in \mathbb N$ be a natural number. For every topological pair of spaces $(X,A)$, we define the $k$-th \v{C}ech homology group $\check{H}_k(X,A)$ as an inverse limit of singular homology groups in the following way.
\begin{align*}
 \check{H}_k(X,A)&:=\varprojlim\left\{H_{k}(X,U)\,|\,  A \subset U \text{ is an open neighbourhood of } A \text{ in } X \right\}
\end{align*}
\end{defi}
Piecing this definition together with the fact that the diagonal was added to the subspace of the symmetric square of a topological pair of spaces in \autoref{sys}, we realize that we can work in homology relative to neighbourhoods of the diagonal every time we are dealing with the relative \v{C}ech homology of the symmetric square. This will turn out to be very useful in a lot of cases.
\begin{rem}[Symmetric Squaring gives Half of the Square]
Let $(X,A)$ be a pair of topological spaces and let $k\in\mathbb N$ be even. Looking at
\begin{align*}
 H_k(X,A)&\rightarrow H_k(X,A)\times H_k(X,A)\overset{\times}{\rightarrow}H_{2k}((X,A)\times (X,A))\overset{pr}{\rightarrow}\check{H}_{2k}(\sys{(X,A)})\\
\sigma=\sum_{i=1}^{n}g_i\sigma_i&\mapsto\sum_{i=1}^{n}g_i\sigma_i\times\sum_{i=1}^{n}g_i\sigma_i \mapsto\sum_{1\leq i,j\leq n}g_i g_j (\sigma_i\times \sigma_j)\mapsto\sum_{1\leq i,j\leq n}g_i g_j pr(\sigma_i\times \sigma_j)
\end{align*} where $\times$ is the outer cross product induced by the simplicial cross product, we see that the result of first squaring and then projecting is twice the image of the symmetric square:
\[ 
2\cdot\sum\limits_{\genfrac{}{}{0pt}{}{i<j}{1\leq i,j \leq n}}{g_i g_j(pr)(\sigma_i \times \sigma_j)}=\sum\limits_{1\leq i,j\leq n}g_i g_j pr(\sigma_i\times \sigma_j).\]
To see that this is true, we use two facts:
\begin{itemize}
 \item $pr(\sigma_i\times\sigma_j)=pr(\sigma_j\times\sigma_i)$ in even dimensions as noted earlier.
 \item Homology groups are taken relative to neighbourhoods of the diagonal here, so by subdivision of simplices\footnote{The standard way of subdividing simplices into smaller simplices is the so-called barycentric subdivision, explained for example in Section III.6 in \cite{MR1335915}.} terms of the form $\sigma_i\times\sigma_i$ can be omitted without changing the sums in homology.
\end{itemize}
Coefficients are taken to be integral in this remark. Using $\mathbb Z_2$-coefficients gives zero as soon as the multiplication by two takes place, so in that case the above only yields that the projection of crossed chains of the form $pr(\sigma \times \sigma)$ is zero in homology with $\mathbb Z_2$-coefficients.
\end{rem}
\v{C}ech homology is shown to be isomorphic to singular homology for many spaces in Proposition 13.17 in \cite{MR1335915}, namely for Euclidean Neighbourhood Retracts. A more detailed discussion on this can be found in \autoref{cechbord}.\\
The special form of \v{C}ech homology is used in the proof of the following theorem as well, while showing how the simplicial symmetric squaring map from \autoref{hom} induces a well-defined map in \v{C}ech homology. This is Theorem 2.7 in \cite{diplom}.

\begin{thm}[Symmetric Squaring is well-defined]\label{symhom}
 Let $k\in \mathbb N$ be even and let $(X,A)\subset W$ be a compact pair and a subset of a smooth manifold $W$. Then the symmetric squaring map from \autoref{hom} induces a well-defined map
\[
 \sys{(\hspace{0,5em}\cdot\hspace{0,5em})}\colon \check{H}_{k}(X,A;\mathbb Z)\rightarrow \check{H}_{2k}\left(\sys{(X,A)};\mathbb Z \right)
\] in \v{C}ech homology.
\end{thm}

\subsection{Properties and Usage}
As noted before, the most important property of the symmetric squaring is that it 'maps fundamental classes to fundamental classes' as soon as homological symmetric squaring of compact smooth manifolds is considered.\footnote{Compare Theorem 3.1 in \cite{borsuk} for $\mathbb Z_2$-coefficients and Proposition 2.13 in \cite{diplom} for $\mathbb Z$-coefficients.} To be able to see that this statement is a sensible one, the term fundamental class has to be explained in the context of \v{C}ech homology of the symmetric square of a manifold.\\ It is known what the term fundamental class means for singular homology of oriented compact manifolds. As \v{C}ech homology of the symmetric square is defined as the inverse limit of singular homology groups relative to neighbourhoods of the diagonal, the fundamental class will be defined as a limit element with respect to neighbourhoods of the diagonal as well. In this context we will regularly use the following notation.
\begin{note}[Reduced by a neighbourhood of the diagonal]\label{note:redneigh}
Let $X$ be a topological space and $X\times X$ its cartesian product. Choose a neighbourhood $U$ of the diagonal $\Delta\subset U\subset X\times X$ in the cartesian product.\\
As soon as '$X\times X$ reduced by a neighbourhood of the diagonal' is mentioned in the following, it shall denote the space $X\times X\setminus U$.
\end{note}

More precisely, for a smooth oriented compact manifold $(B,\partial B)$ of even dimension $k$, we look at the symmetric square $\sys{(B,\partial B)}:=(B\times B/\tau,(\partial(B\times B)\cup \Delta)/\tau)$ and we want to find a reasonable definition of a \v{C}ech fundamental class of this object. For this purpose, we consider all neighbourhoods $V$ of $(\partial(B\times B)\cup\Delta)$\footnote{Note that these $V$ are always neighbourhoods of the diagonal $\Delta$ in $B\times B$.} in $B\times B$, such that $B\times B\setminus V$ and its projection with respect to the coordinate switching map $\tau$ are smooth compact manifolds with boundary.\footnote{About the existence of such compare \autoref{diagsec}.} These 'nice' neighbourhoods are cofinal\footnote{For a definition of cofinality see Section VIII.5 in \cite{MR1335915}.} in all neighbourhoods considered in the inverse limit that defines the \v{C}ech homology group $\check{H}(\sys{(B,\partial B)},\mathbb Z)$, since for every neighbourhood $U$ in that limit we can find a neighbourhood $V$ of the  required form such that $V\subset U$.\\
Every neighbourhood $V$ as above gives rise to an ordinary fundamental class 
\[
\sigma^{V}_{\mathbf f}\in H_{2k}\left(\left(\left(B\times B\right)\text{\textbackslash{}} V\right)/\tau,\partial\left(\left(\left(B\times B\right)\text{\textbackslash{}} V\right)/\tau \right) ,\mathbb Z\right)
\] of the manifold coming from the product $B\times B$ reduced by a neighbourhood of the diagonal. These fundamental classes are the ones that the relative fundamental class of the symmetric square in \v{C}ech homology is made of.

\begin{thm}[Behaviour with respect to Fundamental Classes]\label{fundquadrat}
Let $k\in \mathbb N$ be even. Furthermore, let $(B,\partial B)$ be a $k$-dimensional compact smooth oriented manifold with possibly empty boundary $\partial B$ and let $\sigma_{\mathbf f}\in H_k(B,\partial B,\mathbb Z)$ be its unique fundamental class. The image of this fundamental class under the homological symmetric squaring map is denoted by $\sys{\sigma_{\mathbf f}}$.\\
Then $\sys{\sigma_{\mathbf f}}\in \check{H}_{2k}(\sys{(B,\partial B)})$ is the fundamental class of $\sys{(B,\partial B)}$ in the following sense.\\ For every neighbourhood $U$ of the diagonal in $B\times B$ that appears in the inverse limit defining the group $\check{H}_{2k}(\sys{(B,\partial B)})$ choose a neighbourhood $V\subset U$ as above and consider the corresponding fundamental class $\sigma^{V}_{\mathbf f}\in H_{2k}\left(\left(\left(B\times B\right)\text{\textbackslash{}} V\right)/\tau,\partial\left(\mathbf - \right) ,\mathbb Z\right)$. This can be mapped by inclusion to $i(\sigma^{V}_{\mathbf f})\in H_{2k}\left(\sys{B},U,\mathbb Z\right)$ and can thus be regarded as an element that appears in the \v{C}ech homology of the symmetric square since
\[
\check{H}_{2k}(\sys{(B,\partial B)}) \subset \prod_U H_{2k}\left(\sys{B},U,\mathbb Z\right).\footnote{This follows from a theorem concerning the form of inverse limits in some categories, compare \cite{MR1335915}.}
\]
The symmetric squaring map in homology maps the fundamental class $\sigma_{\mathbf f}$ to the class $\sys{\sigma_{\mathbf f}}\in \check{H}_{2k}(\sys{(B,\partial B)})$ which has the property that
\[
 p(\sys{\sigma_{\mathbf f}})=i(\sigma^{V}_{\mathbf f}) \text{ for all neighbourhoods } U \text{ in the limit,}\footnote{Here we implicitly use the fact that for $V\subset V^{\prime}\subset U$ chosen as above, the inclusions induce maps that send the fundamental class $\sigma^{V}_{\mathbf f}$ to the fundamental class $\sigma^{V^{\prime}}_{\mathbf f}$. This is a property of the inverse limit again, compare \cite{MR1335915}.}
\] where $p$ denotes the projection onto the factor corresponding to $U$ in the inverse limit group $\check{H}_{2k}(\sys{(B,\partial B)})$.
\end{thm}

What makes this property valuable is that it has been used in a proof of a generalised Borsuk-Ulam Theorem in \cite{borsuk}, namely it is a key ingredient in the proof of the main Theorem 2.4 in the just cited reference. The classical Borsuk-Ulam antipodes Theorem states that for all continuous maps $f\colon S^n\rightarrow \mathbb R^n$ there exist antipodal points that are mapped to the same point. In other words it says that the solution set $\left\{v\in S^n|f(v)=f(-v)\right\}$ is not empty for continuous maps $f\colon S^n\rightarrow \mathbb R^n$.\\
In \cite{borsuk}, among other things, solution sets of more complicated forms, such as $\left\{(w,v)\in W\times S^n|F(w,v)=F(w,-v)\right\}$ for parametrised Borsuk-Ulam situations are studied with respect to their homological properties. Here $F\colon W\times S^n\rightarrow \mathbb R^n$ is a continuous map and can be thought of as a family of Borsuk-Ulam maps parametrised by the compact manifold $W$.\\
The main Theorem 2.4 in \cite{borsuk} contains a statement concerning a similar but more general solution set. More precisely, it is proven that the homology group of such a more general solution set contains an element that is mapped to the fundamental class of the compact manifold $W$.\\
This brief explanation can give an idea how \autoref{fundquadrat} can be a useful tool in the cited proof. It inserts a map that is known to map fundamental classes to fundamental classes into a setting where the existence of an element being mapped to a fundamental class is to be proven. That is why in the proof of the generalised Borsuk-Ulam Theorem in \cite{borsuk} the construction of the desired homology class is done by forming the natural intersection pairing\footnote{The existence of such is a special property of \v{C}ech homology, compare Section VIII.13 in \cite{MR1335915}} of a homology class with a symmetric squared one.\\
The motivation for proving this type of generalised Borsuk-Ulam Theorem originally comes from game theory, namely from the theory of games for two players with incomplete information, which are repeated infinitely often as introduced in \cite{MR1342074}. As explained in \cite{borsuk}, the parametrised Borsuk-Ulam Theorem can be used to show the existence of equilibria in games of this type.\\
For reasons of readability, definitions and theorems in this section were formulated using integer coefficients only. But with coefficients in $\mathbb Z_2$ they remain true nonetheless.
\begin{rem}[$\mathbb Z_2$-coefficients]
For $\mathbb Z_2$-coefficients the symmetric squaring in \v{C}ech homology is defined in the same way as above. Omitting coefficients $g_i$ in \autoref{hom} of the simplicial map, \autoref{symhom} and \autoref{fundquadrat} hold for $\mathbb Z$-coefficients replaced by $\mathbb Z_2$-coefficients without the restriction on the dimension or on orientability.\footnote{These facts are proven in \cite{borsuk}.}
\end{rem}

\section{Bordism and \v{C}ech bordism}

This section contains two parts. In the first part, the definition and the most important properties of singular bordism are reviewed without proofs. References mainly are \cite{MR548463,3110162369}. For a more detailed introduction into various bordism theories the reader may also consult \cite{MR0248858}.\\
The second part contains the definition of the so called \v{C}ech bordism. It receives its name from its similarity to \v{C}ech homology and it is used later for generalising the symmetric squaring map to bordism for the same reasons it was used in homology in connection with symmetric squaring in the first place.

\subsection{Bordism}
Since our aim is to look at the symmetric squaring construction in the setting of bordism, we first need to clarify which of the various bordism theories we are going to use. The following definitions can be found in \cite{MR548463} or \cite{3110162369}.

\begin{defi}[oriented singular manifold]
Let $(X,A)$ be a pair of topological spaces. A smooth compact oriented $n$-manifold with boundary $(M,\partial M)$ together with a continuous map $f\colon(M,\partial M)\rightarrow(X,A)$ is called a singular oriented $n$-manifold in $(X,A)$ and is denoted by $(M,\partial M;f)$.
\end{defi}

While singular homology works with maps from simplices to topological spaces, singular bordism deals with maps from smooth manifolds to topological spaces. On smooth manifolds, the equivalence relation called 'bordant' was first introduced by Thom in \cite{MR0061823} and was called 'cobordant'. Later Atiyah\footnote{Compare \cite{MR0126856}.} suggested to distinguish 'cobordism' and the dual construction 'bordism' in analogy to the concepts of cohomology and homology. This is the most common way of denoting these concepts until today.\\
The rough idea introduced by Thom of two compact oriented $n$-manifolds $M_0$ and $M_1$ without boundary being bordant is that their disjoint union $M_0\sqcup -M_1$ is diffeomorphic via an orientation preserving diffeomorphism to the boundary of a compact manifold of dimension $n+1$. Here $-M_1$ is meant to be the manifold $M_1$ with the opposite orientation. This defines an equivalence relation on the oriented diffeomorphism classes of closed oriented  $n$-manifolds. The set of equivalence classes is denoted by $\Omega_n$ and it is a group with addition induced by disjoint union.\footnote{Compare Theorem 2.1 in\cite{MR548463}.}\\
This concept gives rise to the following definition of oriented singular $n$-manifolds with boundary being bordant.

\begin{defi}[bordism, bordant] \label{defbord}
A bordism between two singular oriented $n$-manifolds $(M_0,\partial M_0;f_0)$ and $(M_1,\partial M_1;f_1)$ is a triple $(B,\partial B;F)$ which satisfies
\begin{enumerate}
	\item $B$ is a compact oriented $(n+1)$-manifold with boundary.
	\item The boundary $\partial B$ with the induced orientation of $B$ is the union of three manifolds with boundary $\partial B=M_0\cup -M_1\cup M^{\prime}$ such that $\partial											  M^{\prime}=\partial M_0 \sqcup \partial (-M_1)$ and $M_0\cap M^{\prime}=\partial M_0$ as well as $-M_1\cap M^{\prime}=\partial (-M_1)$.
	\item The continuous map $F\colon B\rightarrow X$ is equal to $f_i$ on $M_i$, i.e. $F_{|{M_{i}}} = f_i$.
	\item The manifold $M^{\prime}$ is mapped to $A$, i.e. $F(M^{\prime})\subset A$.
\end{enumerate}
The oriented singular manifolds $(M_0,\partial M_0;f_0)$ and $(M_1,\partial M_1;f_1)$ are called bordant if there exists a bordism between them. If a singular manifold is bordant to the empty manifold it is said to be zero-bordant.
\end{defi}

The picture below can serve as an illustration of the definition. But it is important to note that the manifold in the picture has to be thought of as being a solid smooth object as soon as manifolds with nonempty boundary are to be considered.
\bigskip

\[ \xy
(0,-3)*\ellipse(3,1){.};
(0,-3)*\ellipse(3,1)__,=:a(-180){-};
(-3,6)*\ellipse(3,1){-};
(3,6)*\ellipse(3,1){-};
(-13,12)*{M_0};
(-13,-6)*{M_1};
(-7,3)*{B};
(-3,12)*{}="1"; 
(3,12)*{}="2"; 
(-9,12)*{}="A2";
(9,12)*{}="B2"; 
"1";"2" **\crv{(-3,7) & (3,7)};
(-3,-6)*{}="A";
(3,-6)*{}="B"; 
(-3,1)*{}="A1";
(3,1)*{}="B1"; 
"A";"A1" **\dir{-};
"B";"B1" **\dir{-}; 
"B2";"B1" **\crv{(8,7) & (3,5)};
"A2";"A1" **\crv{(-8,7) & (-3,5)};
(40,3)*{}=XA;
"XA"*{(X,A)};
{\ar^{f{_0}} (10,12)*{};(35,4)*{}};%
{\ar^{F} (10,3)*{};"XA"};%
{\ar_{f{_1}} (6,-5.7)*{};(35,2)*{}};%
\endxy\]
\bigskip

This equivalence relation reduces to the cobordism relation defined by Thom in \cite{MR0061823} as soon as the topological space considered is the one point space and the occurring manifolds have empty boundary.

\begin{rem}It can be shown that the bordism relation defined in \autoref{defbord} is an equivalence relation on the oriented singular manifolds in $(X,A)$.\footnote{Compare Theorem VIII(13.1) in \cite{3110162369}.} As before, there can be introduced a group structure on the set of equivalence classes of oriented singular $n$-manifolds in $(X,A)$ with addition induced by disjoint union.  The resulting groups are denoted by $\Omega_n(X,A)$ and are called the $n$-th relative bordism groups of $(X,A)$. Furthermore, we denote the oriented bordism class of $(M,\partial M;f)$ by $\left[M,\partial M;f\right]$.
\end{rem}

Singular bordism is similar to singular homology also in the matter of functoriality. Specifically, $\Omega_{\ast}$ can be regarded as a functor from the category of topological pairs of spaces to the category of abelian groups, which gives rise to a generalised homology theory satisfying all Eilenberg-Steenrod axioms for homology\footnote{These are stated in \cite{MR0050886} in Section I.3.} with the exception of the dimension axiom. To understand this we first have to see how the functor $\Omega$ is defined on morphisms and how the boundary operator for the homology theory is induced.\\
To a map $\varphi \colon(X,A)\rightarrow(X_1,A_1)$ there is associated a natural homomorphism $\varphi_{\ast}\colon\Omega_n(X,A)\rightarrow\Omega_n(X_1,A_1)$ given by $\varphi_{\ast}\left[M,\partial M;f\right]=\left[M,\partial M;\varphi\circ f\right]$.  Furthermore, the assignment $\left[M,\partial M;f\right]\mapsto\left[\partial M,\emptyset;f_{|{\partial M}} \right]$ induces a well-defined boundary homomorphism $\partial\colon \Omega_n(X,A)\rightarrow\Omega_{n-1}(A).$

\begin{thm}[Theorem (5.1) in \cite{MR548463}]
On the category of pairs of topological spaces and maps of pairs the bordism functor $\left\{\Omega_\ast(X,A),\varphi_\ast,\partial\right\}$ satisfies six of the Eilenberg-Steenrod axioms for a homology theory. However, for a single point $p$ we have $\Omega_n(p)\simeq \Omega_n$, the oriented Thom bordism group. This means that the bordism functor fails to satisfy the dimension axiom and this makes it a so-called generalised homology theory.
\end{thm}

All definitions and theorems in this section were given using oriented manifolds. Of course, bordism can be studied for unoriented manifolds as well. The bordism relation is defined in the same way as seen in \autoref{defbord} except that there are no requirements on orientability needed.

\begin{rem}\label{unoriented}
The unoriented bordism relation is an equivalence relation as well and the unoriented relative bordism group of a topological pair $(X,A)$ in dimension $n$ is denoted by $\mathcal{N}_n(X,A)$. It defines a functor just as oriented bordism does and satisfies the Eilenberg-Steenrod axioms for a homology theory except the dimension axiom. The unoriented relative bordism groups have been determined for all $CW$-pairs $(X,A)$ in terms of the homology groups of these $CW$-pairs, compare Theorem 8.3 in \cite{MR548463}.
\end{rem}

\input{cech-bordism.tex}

%% file: cech-bordism.tex
\subsection{\v{C}ech-bordism}\label{cechbord}

The reasons for using \v{C}ech homology as a homology theory in connection with the symmetric squaring construction that were discussed in the first section of this chapter were:
\begin{itemize}
 \item Because of the special way \v{C}ech homology is defined as an inverse limit with respect to neighbourhoods, it is always possible to work relative to the diagonal when the homology of the symmetric square is examined.
\item There exists a natural intersection pairing of \v{C}ech homology classes.
\item \v{C}ech homology is isomorphic to singular homology for Euclidean Neighbourhood Retracts.
\end{itemize}
The first and the last of these properties are still useful when transported to the world of bordism. In this section we will see why this is the case and how these properties are true for a theory of \v{C}ech bordism, which is to be defined in this section as well.\\
Our aim is to construct a well-defined symmetric squaring map in (un)oriented bordism, namely \[\sys{(\hspace{0,5em}\cdot\hspace{0,5em})}\colon \mathcal{N}_n(X,A)\rightarrow\mathcal{N}_{2n}( \sys{(X,A)})\text{ and }\sys{(\hspace{0,5em}\cdot\hspace{0,5em})}\colon \Omega_n(X,A)\rightarrow\Omega_{2n}( \sys{(X,A)})\]
\begin{rem}[Oriented vs. Unoriented Bordism]\label{rem:readability}For reasons of readability we will mostly stick to the case of the oriented bordism functor $\Omega$ in this subsection. Nonetheless, the results of this subsection are true for unoriented bordism as well and are proven analogously. In Chapter 2, where the constructed symmetric squaring map is discussed in detail, we will distinguish strictly between the two cases of oriented and unoriented bordism.
 \end{rem}
In order to lift the symmetric squaring construction to bordism, there is to be assigned a singular $2n$-manifold $\sys{\left[M,\partial M;f\right]}\in \Omega_{2n}( \sys{(X,A)})$ to a given singular $n$-manifold $\left[M,\partial M;f\right]\in \Omega_n(X,A)$. Symmetric squaring as defined in \autoref{sys} can be performed on all topological spaces, so it can be performed on smooth manifolds as well. Furthermore, a map $f\colon (M,\partial M)\rightarrow(X,A)$ induces a mapping $\sys{f}\colon\sys{(M,\partial M)}\rightarrow\sys{(X,A)}$ via the assignment $\sys{f}\left[m_1,m_2\right]=\left[f(m_1),f(m_2)\right]$. That is why it seems to be a natural choice to define \[\sys{\left[M,\partial M;f\right]}:=\left[\sys{(M,\partial M)};\sys{f}\right]
\]at first glance.
But unfortunately it is not that easy. Since the involution $\tau$ has the diagonal $\Delta\subset M\times M$ as a fixed point set, $\sys{(M,\partial M)}$ cannot be given the structure of a smooth manifold there coming from the smooth structure of $M$. The symmetric squaring image object suggested above would therefore not be an element in 
$\Omega_{2n}( \sys{(X,A)})$. As in the case of homology the way out is rather looking at an inverse limit of bordism groups, which can be thought of as a \v{C}ech-version of singular bordism. 
So there has to be made a slight change of the formulation of our aim:
We wish to construct a well-defined map \[\sys{(\hspace{0,5em}\cdot\hspace{0,5em})}\;\colon \Omega_n(X,A)\rightarrow\check{\Omega}_{2n}( \sys{(X,A})),\]
where $\check{\Omega}_{2n}( \sys{X,A})$ is defined to be an inverse limit involving neighbourhoods of the diagonal, namely
\[
\varprojlim\left\{\Omega_{2n}(\sys{X},U_{\Delta})\,|\,U_{\Delta} \text{ is an open neighbourhood of } pr(X\times A\cup A\times X\cup\Delta)  \right\}
\] and an analogous unoriented bordism version of this.\\Being able to work relative to a neighbourhood of the (projected) diagonal in $X\times X/\tau$ can be thought of as working relative to a neighbourhood of the diagonal in $M\times M$ as well, because of this remark from Section 5 in \cite{MR548463}.
\begin{rem}\label{relbord}
 Let $V^n\subset M^n$ be a compact regular $n$-submanifold with boundary in a compact $n$-manifold $M^n$ without boundary. If $f\colon M^n\rightarrow X$ is a map with $f(M^n\setminus V^{\circ})\subset A$, then $\left[M^n,f\right]=\left[V^n,f_{|V^n}\right]$ in $\Omega_n(X,A)$.\footnote{We use $(-)^{\circ}$ to denote the interior.}
\end{rem}
Roughly speaking, we can think of $(V^n,f_{|V^n})$ here as being a singular squared manifold with a neighbourhood of the diagonal removed. If it can be assured that the neighbourhood of the diagonal of the squared manifold is mapped by $f$ to the neighbourhood of the diagonal of the squared space, regarding \v{C}ech bordism groups of the squared manifold is as good as looking at \v{C}ech bordism groups of the squared manifold reduced by a neighbourhood of the diagonal.\\
It will be shown later that there is a canonical way to transport a definition of symmetric squaring into the above mentioned inverse limit setting. However, first there is more to say about this construction that we shall call \v{C}ech bordism. 
\begin{defiprop}[(Un)oriented \v{C}ech bordism]\label{cechbordism}
Let $n\in \mathbb N$ be a natural number. For every topological pair of spaces $(X,A)$, we define the limit groups $\check{\Omega}_n(X,A)$ and $\check{\mathcal{N}}_n(X,A)$ as follows.
\begin{align*}
 \check{\Omega}_n(X,A)&:=\varprojlim\left\{\Omega_{n}(X,U)\,|\,  A \subset U \text{ is an open neighbourhood of } A \text{ in } X \right\}\\
\check{\mathcal{N}}_n(X,A)&:=\varprojlim\left\{\mathcal{N}_{n}(X,U)\,|\,  A \subset U \text{ is an open neighbourhood of } A \text{ in } X \right\}
\end{align*}
This gives \v{C}ech-bordism-functors \[
\check{\Omega}\; , \; \check{\mathcal{N}}\colon \mathsf{Top^2}\rightarrow \mathsf{Grp}
\] which associate groups to topological pairs of spaces in the way noted above.
\end{defiprop}
This is a proposition as well as a definition since it has to be proven that the defined assignments really form functors. In order to prove this, we will first have to define what $\check{\Omega}$ and $\check{\mathcal{N}}$ associate to morphism in the category of topological pairs.
\begin{proof}
Although the notation involves only the oriented bordism groups from now on, everything in this proof is as well true for the unoriented case.\\
 To define how the \v{C}ech-bordism functor associates morphisms in the category of groups to morphisms in the category of topological pairs, we use the universal property of inverse limits.\\ Let $g\colon(X,A)\rightarrow(Y,B)$ be a continuous map between two pairs of topological spaces. This induces a unique map $\check{\Omega}(g)\colon\check{\Omega}(X,A)\rightarrow\check{\Omega}(Y,B)$ as follows.\\ Per definition, the inverse limit $\check{\Omega}(Y,B)$ always comes together with a projection $\pi_V$ for every open neighbourhood $V$ of $B$ such that 
\[\xymatrix{
&\check{\Omega}(Y,B)\ar[rd]_{\pi_V}\ar[ld]^{\pi_{V^{\prime}}}&\\
\Omega(Y,V^{\prime})\ar[rr]^{\Omega(i)} & &\Omega(Y,V) \\
}\] commutes for all $V^{\prime}\subset V$ and inclusions $i\colon V^{\prime}\hookrightarrow V$.\\
Furthermore it has the universal property that for every other such pair $(G,\psi_V)$ of a group $G$ and maps $\psi_V$ from $G$ to $\Omega(Y,V)$ for every $V$ there exists a unique homomorphism $\phi\colon G\rightarrow \check{\Omega}(Y,B)$ that makes the following diagram commutative
\[
\xymatrix{
&G\ar[ddl]_{\psi_{V^{\prime}}}\ar[ddr]^{\psi_V}\ar@{.>}[d]^{\phi}&\\
&\check{\Omega}(Y,B)\ar[rd]_{\pi_V}\ar[ld]^{\pi_{V^{\prime}}}&\\
\Omega(Y,V^{\prime})\ar[rr]^{\Omega(i)} & &\Omega(Y,V)
}
\] for all $V^{\prime}\subset V$.\\
Looking at the universal property with respect to the pair $(\check{\Omega}(X,A),g_V)$, where $g_V$ is defined by \[g_V\colon\check{\Omega}(X,A)\stackrel{\pi_{g^{-1}(V)}}{\rightarrow}\Omega(X,g^{-1}(V))\stackrel{\Omega(g)}{\rightarrow}\Omega(Y,V),\] we see the unique induced map $\check{\Omega}(g)$ in the diagram
\[\xymatrix{
&\check{\Omega}(X,A)\ar[ddl]_{g_{V^{\prime}}}\ar[ddr]^{g_V}\ar@{.>}[d]^{\check{\Omega}(g)}&\\
&\check{\Omega}(Y,B)\ar[rd]_{\pi_V}\ar[ld]^{\pi_{V^{\prime}}}&\\
\Omega(Y,V^{\prime})\ar[rr]^{\Omega(i)} & &\Omega(Y,V).
}
\]
 That \v{C}ech bordism preserves identity morphisms and composition of morphism follows directly from the corresponding properties of bordism using the universality of the inverse limit again.
\end{proof}

\begin{rem}As was already pointed out before, the above definition and notation are derived from the similarity to {\v C}ech-homology as defined in \autoref{cechhom}.
\end{rem}

One reason why it was sensible to work with \v{C}ech homology in connection with the results from \autoref{resultshom} was that \v{C}ech homology is isomorphic to singular homology for a lot of interesting topological spaces. The same is true for \v{C}ech bordism, so in many cases working with {\v C}ech-bordism is as good as working with singular bordism, since the resulting groups are isomorphic. We will prove this now.\\ 
As a preparation to prove the next proposition, we need a definition and a lemma. The proof of \autoref{ret} uses Theorem 3 from \cite{MR0073982} and the technique used in the proof of Proposition IV. 8.6 in \cite{MR1335915}.

\begin{defi}[Euclidean Neighbourhood Retract, ENR]\label{def:ENR}
 A topological space $X$ is called a Euclidean Neighbourhood Retract (ENR) if $X$ is homeomorphic to a subspace $Y\subset \mathbb R^n$, which is a neighbourhood retract, i.e. there exists a neighbourhood $U$ of $Y$ in $\mathbb R^n$ and a retraction $r\colon U\rightarrow Y$ such that $r\circ i_{(Y\hookrightarrow U)}=id_Y$. 
\end{defi}

Examples for spaces which are ENRs are compact manifolds with or without boundary as well as finite $CW$ complexes.\footnote{Compare Corollaries A.9 and A.10 in \cite{MR1867354}.} Dold also proves in Proposition IV.8.10 in \cite{MR1335915} that a Hausdorff space $X$ which is a finite union of ENRs, each of which is open in $X$, is itself an ENR.

\begin{lem}\label{ret}
 Let $X$ and $A\subset X$ be ENRs. Then there exists an open neighbourhood $U_0$ of $A$ and a map $r_0\colon (X,U_0)\rightarrow(X,A)$ such that $r_0$ restricted to $A$ is equal to the inclusion of $A$ into $X$ and the composition
\[
(X,A)\stackrel{i_0}{\hookrightarrow}(X,U_0)\stackrel{r_0}{\rightarrow}(X,A)
\] is homotopic to the identity, i.e. $id_{(X,A)}\simeq r_0\circ i_0$.\\
The neighbourhood $U_0$ of $A$ can be chosen small enough to be contained in any other given neighbourhood $U$ of $A$.
\end{lem}
\begin{proof}
The idea of the proof is to first use a retraction which is given by the fact that $A$ is an ENR, then use the fact that $X$ is an ENR to find a homotopy between the given retraction and the identity on a neighbourhood of $A$. This homotopy can then be extended to the whole of $X$ using a theorem of Dowker.\\
Since $A$ is an ENR, there exists an open neighbourhood $V$ of $A$ in $X$ and a map $r\colon V\rightarrow A$, which is a retraction, so \begin{equation}
 r\circ i_ {(A\hookrightarrow V)}  =id_A.\label{rid} 
\end{equation}
The fact that $X$ is an ENR makes it now possible to find a smaller neighbourhood $W$ of $A$ in $V$ and a homotopy from $r$ to the identity on that neighbourhood $W$ in the following way:
Let $X\stackrel{\iota}{\rightarrow} O\stackrel{\rho}{\rightarrow} X$ be such that $O\subset \mathbb R^n$ is open and $\rho\circ\iota=id_X$, which exists since $X$ is an ENR. Now choose $W\subset V$ to be the set of all points $x\in V$ such that the whole segment from $\iota r(x)$ to $\iota(x)$ lies in $O$ and define a homotopy \begin{align*}H\colon W\times [0,1]&\rightarrow X \text{ by } \\ H(x,t)&=\rho[(1-t)\iota(x)+t\iota r(x)].\end{align*}
 The retraction $r$ was chosen in a way such that $\iota r(x)=\iota(x)$ for all $x\in A$, compare (\ref{rid}) above. So it follows that $A\subset W$. Since $r$ is continuous and $A\subset X$ can be thought of as lying in $\mathbb R^n$ it is also true that $r$ fixes elements of the (topological) boundary of $A$, which means that $\bar{A}\subset W$ as well. This can be seen by letting $a$ be an element of the boundary of $A$, taking a sequence $a_m$ of elements in $A$ that converges to $a$ and computing $r(a)=r(\lim\limits_{m\to\infty} a_m)=\lim\limits_{m\to\infty}r(a_m)=\lim\limits_{m\to\infty}a_m=a.$\\
What was constructed up to now is a map
\begin{align*}
H\colon X\times \left\{0\right\}\cup \bar{A}\times [0,1]&\rightarrow X \text{ with }\\
H(x,0) &=x \text{ for all }x\in X\\
H(x,t)&=x \text{ for all }x\in \bar{A} \text{ and } t\in [0,1]\\
H(x,1)&=r(x) \text{ for all } x\in \bar{A},
\end{align*} 
which can be extended to the set $X\times\left\{0\right\}\cup W\times [0,1]$. Dowker's Theorem 3 in \cite{MR0073982} states that in such cases, the homotopy can be extended to the whole of $X$. If we call this extension $\bar{H}\colon X\times [0,1]\rightarrow X$ than the  first part of the lemma can be proven by simply defining $U_0:=W$ and $r_0(x):=\bar{H}(x,1)$ for all $x\in X$. The set $W$ was chosen above to be a neighbourhood contained in $V$. The retraction $r\colon V\rightarrow A$ would as well work with any other neighbourhood contained in $V$. This is the reason why $U_0$ can be chosen such that it would be contained in any given neighbourhood $U$ of $A$. This proves the lemma.
\end{proof}

\autoref{ret} is a crucial ingredient in the following proof. What makes it so valuable is that it provides pair versions of retractions and we examine mostly maps of pairs in connection with relative bordism or \v{C}ech bordism groups.

\begin{prop}\label{prop:singcech}
 Let $(X,A)$ be such that $X$ is an ENR and $A\subset X$ is an ENR as well. Then \[\check{\Omega}(X,A)\simeq {\Omega}(X,A) \text{ via the map }\]
\begin{align*}
               j:\Omega(X,A)&\rightarrow \check{\Omega}(X,A)\text{ defined by }\\
[B^n,\partial B^n,f]&\mapsto\left\{i^U_\ast\left(\left[B^n,\partial B^n,f\right]\right)\right\}_U\in \check{\Omega}(X,A)\subset \prod_{U\supset A}\Omega(X,U),
              \end{align*}
 where $i^U\colon (X,A)\rightarrow (X,U)$ denotes the inclusion and $i^U_\ast$ is induced by it, i.e.\[
i^U_\ast\left(\left[B^n,\partial B^n,f\right]\right)=\left[B^n,\partial B^n,i^U\circ f\right].
\]
This is a natural transformation of functors.
\end{prop}
\begin{proof}
To prove the isomorphism claimed above, the inverse mapping to $j$ is to be constructed and this is where \autoref{ret} is going to be used. Define
\begin{align*}
 \rho\colon\check{\Omega}(X,A)&\rightarrow \Omega(X,A)\text{ by }\\
\left\{\left[B^n_U,\partial B^n_U, f_U\right]\right\}_U&\mapsto \left[B^n_{U_0},\partial B^n_{U_0}, r_0\circ f_{U_0}\right],
\end{align*}
where $r_0\colon (X,U_0)\rightarrow (X,A)$ is defined as in \autoref{ret}, i.e. $r_0\circ i_{(A\hookrightarrow V)}=id_A$ and $(X,A)\stackrel{i_0}{\hookrightarrow}(X,U_0)\stackrel{r_0}{\rightarrow}(X,A)$ is homotopic to the identity. 
This does not depend on the choice of the neighbourhood $U_0$ and the map $r_0$ because of the special form elements of the limit group $\check{\Omega}(X,A)$ have and due to the fact that the homotopy of \autoref{ret} can be used as a homotopy here as well. Let 
\[
 r_i\colon (X,U_i)\rightarrow (X,A) \text{ for } i=0, 1
\] be two different choices of maps and neighbourhoods with the property stated in \autoref{ret}. For an element of the limit group $\check{\Omega}(X,A)$ we know\footnote{Compare Proposition VIII.5.7 in \cite{MR1335915}.} that for all neighbourhoods $U^\prime\subset U$
\begin{align}\label{limitprop}
\left[B^n_{U^{\prime}},\partial B^n_{U^{\prime}}, i_{(X,U^{\prime})\hookrightarrow (X,U)} \circ f_{U^{\prime}}\right]=\left[B^n_U, \partial B^n_U,f_U\right]. 
\end{align} This property will be used in the following for the open neighbourhood $U_0\cap U_1\subset U_0$ and $U_0\cap U_1\subset U_1$. Compute
\begin{align*}
 \left[B^n_{U_0},\partial B^n_{U_0},r_0\circ f_{U_0}\right]&=\left[B^n_{U_0\cap U_1},\partial B^n_{U_0\cap U_1},r_0\circ i_{(X,U_0\cap U_1)\hookrightarrow(X,U_0)} \circ f_{U_0\cap U_1}\right]\\
&=\left[B^n_{U_0\cap U_1},\partial B^n_{U_0\cap U_1},r_1\circ i_{(X,U_0\cap U_1)\hookrightarrow(X,U_1)} \circ f_{U_0\cap U_1}\right]\\
&=\left[B^n_{U_1},\partial B^n_{U_1},r_1\circ f_{U_1}\right],
\end{align*} where the first and the last equation follow from the limit property (\ref{limitprop}) as noted above. The second equation holds due to the fact that $r_0\circ i_{(X,U_0\cap U_1)\hookrightarrow(X,U_0)}$ and $r_1\circ i_{(X,U_0\cap U_1)\hookrightarrow(X,U_1)}$ are both homotopic to the identity via the homotopy given in \autoref{ret} and thus are homotopic to each other. \footnote{Here and in the sequel we implicitly use the fact that the homotopy axiom is valid for bordism several times.} This shows that the map $\rho$ is well-defined.\\
We can see that $\rho\circ j=id$ as follows.
\begin{align*}
\rho j\left(\left[B^n,\partial B^n, f\right]\right)&=\rho\left(\left\{\left[B^n,\partial B^n,i^U\circ f\right]\right\}_U\right) \\
&=[B^n,\partial B^n,r_0\circ i^{U_0}\circ f]\\
&=\left[B^n,\partial B^n,f\right],
\end{align*}
where the last equation follows from the fact that $r_0 \circ i^{U_0}$ is homotopic to the identity on $(X,A)$. If this homotopy is called $H$, the bordism between $[B^n,\partial B^n,r_0\circ i^{U_0}\circ f]$ and $\left[B^n,\partial B^n,f\right]$ is given by 
\begin{align*}
\tilde{H}\colon B^n\times [0,1]&\rightarrow X\\
(x,t)&\mapsto \tilde{H}(x,t):=H(f(x),t).
\end{align*}
What is left to prove is that $j\rho=id$ as well. Compute
\begin{align*}
j\rho\left(\left\{\left[B^n_U,\partial B^n_U, f_U\right]\right\}_U\right)&=j\left(\left[B^n_{U_0},\partial B^n_{U_0},r_0\circ f_{U_0}\right]\right)\\ 
&=\left\{\left[B^n_{U_0},\partial B^n_{U_0}, i^U \circ r_0\circ f_{U_0}\right]\right\}_U
\end{align*}
The composition $j\rho$ is equal to the identity if for all neighbourhoods $U$ of $A$ the following holds:
\[
\left[B^n_U,\partial B^n_U,f_U\right]=\left[B^n_{U_0},\partial B^n_{U_0}, i^U \circ r_0\circ f_{U_0}\right].
\]
To see this, first fix a neighbourhood $U$ of $A$ and then choose a neighbourhood $V$ of $A$ which lies inside $U\cap U_0$. Property (\ref{limitprop}) is needed again to establish the above equation. Furthermore, the above homotopy is considered again:\\
The homotopy $H$ mentioned before not only gives a homotopy between the identity on $(X,A)$ and $r_0 \circ i^{U_0}$, but also can be thought of as a homotopy between the inclusion $i_{(X,V)\hookrightarrow (X,U)}\colon(X,V)\rightarrow(X,U)$ and the composition $i^U\circ(r_0{|_V})\colon(X,V)\hookrightarrow (X,U_0)\stackrel{r_0}{\rightarrow}(X,A)\stackrel{i^U}{\rightarrow}(X,U)$ because they agree as maps defined just on $X$ respectively.\\
This enables us to compute
\begin{align*}
 \left[B^n_U,\partial B^n_U,f_U\right]
&=\left[B^n_V,\partial B^n_V,i_{(X,V)\hookrightarrow (X,U)}\circ f_V\right] \text{, which follows from }V\subset U \text{ together}\\
&\text{ with the limit property }\\
&=\left[B^n_V,\partial B^n_V,i^U\circ r_0\circ i_{(X,V)\hookrightarrow (X,U_0)}\circ f_V\right]\text{, using the homotopy just }\\ &\text{ explained }\\
&=\left[B^n_{U_0},\partial B^n_{U_0},i^U\circ r_0\circ f_{U_0}\right] \text{, by the limit property using }V\subset U_0.
\end{align*}
Altogether we have that $\rho j=id$ and $j \rho=id$, which proves the isomorphism.\\
This transformation is natural because of the commutativity of the following diagram for all continuous maps $\varphi\colon(X,A)\rightarrow(Y,B)$ between topological spaces.
\[
\xymatrix{
\Omega(X,A)\ar[r]^{\Omega(\varphi)}\ar[d]_{j}&\Omega(Y,B)\ar[d]^j\\
\check{\Omega}(X,A)\ar[r]_{\check{\Omega}(\varphi)}&\check{\Omega}(Y,B)
}
\]
Commutativity of this diagram can be seen via this straightforward computation, which follows directly from the definitions of the maps involved.
\begin{align*}
 (j\circ\Omega(\varphi))\left[B^n,\partial B^n,f\right]&=j(\left[B^n,\partial B^n,\varphi\circ f\right])\\
&=\left\{i^U_\ast\left(\left[B^n,\partial B^n,\varphi\circ f\right]\right)\right\}_U\\
&=\check{\Omega}(\varphi)\left\{(i^{U}_{\ast}\left[B ,\partial B^n,f \right])\right\}_U\\
&=(\check{\Omega}(\varphi)\circ j)(\left[B ,\partial B^n,f \right])
\end{align*}
This proofs the proposition.
\end{proof}

\begin{rem}
 The analogous result holds in homology, namely for topological pairs $(X,A)$ such that $X$ and $A$ are Euclidean Neighbourhood Retracts, singular homology is isomorphic to \v{C}ech-homology. The preceding proof is based on a proof for this analogous result that can be found in \cite{MR1335915} in Section VIII.13. 
\end{rem}

%% file: onehalfsquaring.tex
\chapter{About manifolds, diagonals and bordism}

The second chapter is devoted to the lifting of the symmetric squaring construction to a well-defined map in unoriented and oriented bordism. Specifically, this means that we have to deal a lot with manifolds, since they are the major objects in bordism. To overcome technical difficulties concerning manifolds with boundary in particular, the first section of this chapter provides a toolkit for manifolds which will be used repeatedly in the following.\\
Two of the main results of this thesis are contained in the second section of this chapter. It is shown that there exist well-defined symmetric squaring maps in unoriented (compare \autoref{halfun}) and oriented bordism (compare \autoref{halfor}). As a preparation for these results, we take a close look on the choice of neighbourhoods of the diagonal in squared manifolds.\\
As we then will have provided ways of interpreting the symmetric squaring map in homology and in bordism, we can also give a passage between those worlds. There is a canonical way of mapping from bordism to homology called the fundamental class transformation.\footnote{Compare \cite{MR548463} Section 6.} We prove in \autoref{compat} that the symmetric squaring maps defined in homology and bordism are compatible with this canonical passage.

\section{Toolkit for manifolds}

Common technical issues concerning smooth manifolds have to do with nonempty boundaries, with corners, that have to be smoothened or with the question whether subspaces of the manifold inherit a smooth structure and then are smooth manifolds again themselves.\\
Dealing with bordism will always have to do with manifolds with nonempty boundary. Furthermore, the squaring performed during the process of symmetric squaring produces manifolds with some sort of corners as soon as it is done to manifolds with nonempty boundary. So two of the named issues occur naturally in our studies. The third issue comes into play as soon as we consider squared manifolds reduced by a neighbourhood of the diagonal. However, these are the natural objects to look at in connection with \v{C}ech bordism as was noted earlier.\\
The toolkit we want to give here is chosen to provide the necessary technical knowledge to overcome exactly these issues.

\subsection{Double of a manifold}\label{subsec:double}
Almost all manifolds occurring here are manifolds with nonempty boundary. This sometimes leads to technical problems near the boundary. A useful concept to avoid these is that of the double of a manifold. The next paragraph refers to Section I.5 in \cite{MR0198479}.

\begin{defi}\label{doubledef}
 Let $N$ be a smooth $n$-dimensional manifold with nonempty boundary. The double of $N$ is the union of $N_0:=N\times 0$ and $N_1=N\times 1$, with $(x,0)$ and $(x,1)$ identified whenever $x$ is in $\partial N$. It is denoted by $D(N)$.
\end{defi}
Since in the definition of the double of a manifold, a manifold is attached to a copy of itself along the boundary, the resulting object does not have a boundary anymore, but does not have any more or less information than the single manifold had before. It can be shown that the resulting double is a smooth manifold itself.
\begin{lem}\label{double}
The double $D(N)$ can be equipped with a differentiable structure.
\end{lem}
\begin{proof}
 Let $\kappa_0\colon U_0\rightarrow\partial N_0\times\left[0,1\right)$ and $\kappa_1\colon U_1\rightarrow\partial N_1\times\left(-1,0\right]$ be collars of $\partial N_0$ and $\partial N_1$ respectively, i.e. diffeomorphisms of neighbourhoods $U_i$ of the boundaries $\partial N_i$ in $N_i$, which identify $\partial N_i$ with $\partial N_i\times 0$ for $i=0,1$.\footnote{The existence of collars is proven in the so-called Collaring Theorem, compare for example Theorem (13.6) in \cite{MR674117}. We will make constant use of this from now on.\label{footnote:collar}} Let $U$ be the union of $U_0$ and $U_1$ in $D(N)$ and let $\kappa\colon U\rightarrow \partial N\times \left(-1,1\right)$ be the homeomorphism induced by $\kappa_o$ and $\kappa_1$. We get a well-defined  smooth structure on $D(N)$ if we require $\kappa$ to be a diffeomorphism and the inclusions of $N_0$ and $N_1$ in $D(N)$ to be smooth imbeddings.\\
The second condition allows us to use the original charts of the smooth manifold $N$ for $D(N)$ in an appropriate distance from the identified points $\partial N_i$ with $i=0,1$, whereas the first condition can be used to produce charts for $D(N)$ around these identified points using boundary charts of $N$. Namely, let $p\in \partial N$ and $\varphi_\partial:V\rightarrow \mathbb R^{n-1}$ be a chart of $\partial N$ around $p$. Then a chart for $D(N)$ around $p$ is given by
\[
\Phi\colon\kappa^{-1}\left(V\times\left(-1,1\right)\right)\xrightarrow{\kappa}V\times\left(-1,1\right)\xrightarrow{\varphi_\partial\times id}\mathbb{R}^n.
\]
This structure is also built in a way such that restricting to $\kappa_i$ above gives back ordinary boundary charts of the manifold $N$.
\end{proof}
Orientations were not involved in the preceding definition and lemma. It is good to know, however, that the doubling construction works with oriented manifold as well. Furthermore, the result of doubling a manifold is unique up to diffeomorphisms.
\begin{rem}[orientation, uniqueness]\label{rem:colluni}
\begin{itemize}
 \item[]
 \item Oriented manifolds can be doubled in an analogous manner, the only difference being that in \autoref{doubledef} $N_0$ is taken to be $N_0= -N \times 0$, where $-N$ denotes the manifold $N$ with the opposite orientation as $N$. In \autoref{double} the collars then have to be chosen to be orientation-preserving.
 \item The uniqueness of this structure up to diffeomorphism is shown in \cite{MR0198479} in Theorem I.6.3.
\end{itemize} 
\end{rem}

\subsection{Straightening the angle}

The first step of the symmetric squaring procedure always involves the squaring of a space. To perform this with a smooth manifold $W$ with nonempty boundary can cause difficulties since the parts $\partial  W\times \partial W$ of the boundary of $W\times W$ do not carry the structure of a smooth boundary in a canonical way. The usual way out is a construction called straightening the angle,  which is described in \cite{MR548463}, \cite{MR674117} or \cite{Mil59}. The first of these contains the following description.\\
Let $\mathbb R_+\subset \mathbb R$ denote the set of all non-negative real numbers. Pick a homeomorphism
\[
\alpha\colon\mathbb R_+\times \mathbb R_+\rightarrow \mathbb R\times \mathbb R_+,
\]which is a diffeomorphism of $\mathbb R_+\times \mathbb R_+\smallsetminus (0,0)$ onto $\mathbb R\times \mathbb R_+\smallsetminus (0,0)$. Take for example $\alpha(r,\theta)=(r,2\theta)$ for $0\leq\theta\leq\pi/2 $ in polar coordinates.
\begin{lem}\label{straight}
 Let $W$ be an $n$-dimensional topological manifold and $M\subset W$ an $(n-2)$-dimensional submanifold without boundary, closed in $W$, such that
\begin{itemize}
 \item $W\smallsetminus M$ has a differentiable structure.
 \item $M$ has a differentiable structure.
 \item There is an open neighbourhood $U$ of $M$ in $W$ and a homeomorphism $\Phi$ of $U$ onto $M\times \mathbb R_+\times \mathbb R_+$ with $\Phi(x)=(x,0,0)$ for all $x\in M$ and where $\Phi$ is a diffeomorphism from $U\smallsetminus M$ onto $M\times \mathbb R_+\times \mathbb R_+\smallsetminus M\times 0\times 0$.
\end{itemize}
Then there exists a differentiable structure on $W$ that induces differentiable structures on $W\smallsetminus M$ and $U$. In general, this can be used to introduce a differentiable structure on manifolds that arise from glueing two differentiable manifolds together. 
\end{lem}
\begin{proof}
 Let 
\[
\alpha^\prime\colon M\times \mathbb R_+\times \mathbb R_+\rightarrow M \times \mathbb R\times \mathbb R_+
\]be given by $\alpha^\prime(x,y,z)=(x,\alpha(y,z))$. Then the composition $\alpha^\prime\circ\Phi\colon U\rightarrow M\times \mathbb R\times \mathbb R_+$ is a homeomorphism and there is a product differentiable structure on $U$ such that $\alpha^\prime\circ\Phi$ is a diffeomorphism.\footnote{This is achieved with the help of collars, for details see the appendix in\cite{Mil59}.} Hence $U$ and $W\smallsetminus M$ have differentiable structures that coincide on their intersection, so there is a differentiable structure on $W$ which induces the differentiable structures of $U$ and $W\smallsetminus M$.
\end{proof}
The most common case where this is used is the one already mentioned above, namely the product of two smooth manifolds with nonempty boundaries. If $W_1\times W_2$ is of that kind, then $W_1\times W_2\smallsetminus \partial W_1\times \partial W_2$ has a natural differentiable structure that comes from the smooth structures of $W_1$ and $W_2$. But the part $\partial W_1\times \partial W_2$ does not because the charts coming as a cross product from the boundary charts of $W_1$ and $W_2$ do not map into a halfplane $\mathbb R\times \mathbb R_+$ but a into a 'quarterplane' $\mathbb R_+\times \mathbb R_+$. The idea of straightening the angle is to take the diffeomorphism $\alpha$ from the 'quarterplane' to the halfplane and compose it with the crossed charts.\\ By the Collaring Theorem\footnote{Compare for example Theorem (13.6) in \cite{MR674117}.} there exists a homeomorphism from a neighbourhood of $\partial W_1\times\partial W_2$ in $W_1\times W_2$ onto $\partial W_1\times\partial W_2\times \mathbb R_+\times \mathbb R_+$ with the properties in the lemma. So $W_1\times W_2$ can be given a differentiable structure and from now on a manifold $W_1\times W_2$ will always be assumed to have all angles straightened by this construction if necessary.

\subsection{A Riemannian tool}

On every smooth manifold there can be introduced a Riemannian metric using a partition of unity. A Riemannian metric on a smooth manifold is a family of positive definite inner products on the tangential spaces of the manifold depending smoothly on the basepoint of the tangential spaces. This concept is used to define a notion of distance on smooth manifolds. In this section, the squared distance function assigning to points of a manifolds their squared distance from a submanifold will be shown to be a smooth function. In the case where the diagonal is regarded as a submanifold of the square of a manifold, this function will be used later on to find neighbourhoods of the diagonal such that the complement is a smooth manifold with boundary again. Further information on Riemannian metrics and details for the following section can be found in \cite{MR1468735} in Chapter 3.

\begin{lem}\label{8-5}\footnote{This is Exercise 8-5 in \cite{MR1468735}.}
Let $M\subset\tilde{M}$ be a compact, embedded Riemannian submanifold. For $\epsilon>0$ let $N_\epsilon$ denote the subset $\left\{V\colon |V|<\epsilon\right\}$ of the  normal bundle  $NM$ and let $M_\epsilon$ be the set of points in $\tilde{M}$ that have distance less than $\epsilon$ from $M$. From the Tubular Neighbourhood Theorem\footnote{This can be found for example in Chapter IV, \S 5 of \cite{MR1931083}.} it follows that for $\epsilon$ small enough the restriction to $N_\epsilon$ of the exponential map of $\tilde{M}$ is a diffeomorphism from $N_\epsilon$ to $M_\epsilon$. Let $r\colon \tilde{M}\rightarrow \left[0,\infty\right)$ assign to every point in $\tilde{M}$  its distance to $M$. Then $r^2$ is a smooth function on each tubular neighbourhood $M_\epsilon$. 
\end{lem}
\begin{proof}
 Since $exp\colon N_\epsilon\rightarrow M_\epsilon$ is a diffeomorphism for $\epsilon$ small enough, it follows that for every point $p\in M_\epsilon$ there is exactly one element $(q,V)=exp^{-1}(p)\in N_\epsilon$, i.e. there exists exactly one maximal geodesic, that starts from $q\in M$ in the normal direction $V$ to $M$ (with $|V|<\epsilon$) and reaches $p$ at the time $1$. It is also known that every geodesic that minimises the distance from $p$ to $M$ must start normal to $M$. So it follows that the distance from $p$ to $M$ is equal to the distance from $p$ to $q$. In other words
\[
r^2(p)=g(exp^{-1}(p),exp^{-1}(p)),
\] where $g$ is the Riemannian metric. This shows that $r^2$ is smooth as a composition of smooth maps.
\end{proof}

\subsection{Morse functions}

In \cite{MR0190942}, a special sort of Morse functions is described in connection with h-cobordism. Some of the results about the existence of such functions can be used later for the construction of a bordism in the proofs of \autoref{halfun} and \autoref{halfor}. This will be a crucial ingredient needed to be able to prove that symmetric squaring induces a well-defined mapping in bordism.

\begin{defi}
Let $W$ be a compact smooth $n$-manifold such that the boundary $\partial W$ is the disjoint union of two open and closed submanifolds, i.e. $\partial W=V_0\sqcup V_1$. A Morse function on $(W;V_0,V_1)$ is a smooth function $f\colon W\rightarrow\left[a,b\right]$ such that 
\begin{itemize}
 \item $f^{-1}(a)=V_0$ and $f^{-1}(b)=V_1$.
 \item All the critical points of $f$ are interior (lie in $W\smallsetminus\partial W$) and are non-degenerate.
\end{itemize}
\end{defi}

\begin{rem}
Because of Morse Lemma\footnote{See \cite{MR0163331} Lemma I.2.2.}, the critical points of Morse functions can be shown to be isolated and there are only finitely many critical points since $W$ is compact.
\end{rem}

\begin{thm}\label{morse}
There exists a Morse function on every triad $(W;V_0,V_1)$ where $W$ is a compact smooth $n$-manifold such that the boundary $\partial W$ is the disjoint union of two open and closed submanifolds, i.e. $\partial W=V_0\sqcup V_1$ just as in the previous definition.
\end{thm}

\begin{proof}
 This is Theorem 2.5 in \cite{MR0190942} on page 9.
\end{proof}

For our purposes, a slightly more complicated version of \autoref{morse} is needed. The proof, however, is based on the proof of \autoref{morse}.

\begin{thm}\label{meinmorse}
 Let $(M_0,\partial M_0;f_0)$ and $(M_1,\partial M_1;f_1)$ be two bordant singular $n$-manifolds. Furthermore, let $(W,\partial W; F)$ be a bordism between them, which means
\begin{itemize}
			\item $W$ is a compact $(n+1)$-manifold with boundary.
			\item It is $\partial W= M_0\cup M_1\cup M^{\prime}$ with $\partial M^{\prime}=\partial M_0\sqcup \partial M_1$ and $M_i\cap M^{\prime}=\partial M_i$ for $i=0, 1$.
			\item $F_{|{M_i}}=f_i$ for $i=0, 1$.
			\item $F(M^{\prime})\subset A$.
\end{itemize}

Then there exists a Morse function $f\colon W\rightarrow \left[0,1\right]$ such that
\begin{enumerate}
 \item $f$ does not have any critical points in a neighbourhood of $\partial W$,
 \item $f^{-1}(i)=M_i$ for $i=0,1$,
 \item $f_{|M^\prime}$ is a Morse function on the triad $(M^\prime;\partial M_0,\partial M_1)$ that takes distinct values at distinct critical points,
 \item $f$ takes distinct values at distinct critical points of $f$ and additionally no critical point of $f_{|M^\prime}$ has the same value as any critical point of $f$.
\end{enumerate}

\end{thm}

\begin{proof}
 First use \autoref{morse} and Lemma 2.8 on page 17 of \cite{MR0190942} to define a Morse function $\phi\colon M^\prime\rightarrow \left[0,1\right]$ on $(M^\prime;\partial M_0,\partial M_1)$, for which $\frac{1}{2}$ is a regular value and which maps distinct critical points to distinct values. Then choose a collaring neighbourhood $\kappa\colon U \rightarrow \partial W\times \left[0,1\right)$ and define $\bar{f}\colon\partial W\times\left[0,1\right)\rightarrow\left[0,1\right]$ in the following way
\[\bar{f}(x,t)= \begin{cases}
         t & \text{ for } x\in M_0,\\
	 \phi(x)+(1-2\phi(x))\cdot t & \text{ for } x\in M^\prime,\\
	 1-t & \text{ for } x\in M_1.	
         \end{cases}\]
This is well defined, because the definitions agree on the intersections $M_i\cap M^{\prime}=\partial M_i$ due to the fact that $\phi^{-1}(i)=\partial M_i$ for $i=0, 1$. Furthermore, the second desired property holds for $f:=\bar{f}\circ \kappa$. To achieve the first of these, look at the derivative of $\bar{f}$ in $t$-direction. It is
\[
 \frac{\partial \bar{f}}{\partial t}(x)=\begin{cases}
                               1 & \text{ for } x\in M_0,\\
			       1-2\phi(x) & \text{ for } x\in M^\prime,\\
			       -1 & \text{ for } x\in M_1
                              \end{cases}
\]which is only zero for points where $\phi(x)=\frac{1}{2}$. But $\phi$ is assumed to have $\frac{1}{2}$ as a regular value, so that we can conclude that there exists a neighbourhood of the boundary $\partial W$ on which $f$ has only regular points. Using a partition of unity, this function can be extended to a Morse function $f\colon W\rightarrow \left[0,1\right]$ satisfying the first three of the desired properties.\footnote{For details, consult \cite{MR0190942}.} After that, this function can be altered slightly with the help of Lemma 2.8 on page 17 of \cite{MR0190942} to fulfil all conditions needed.
\end{proof}

\section{Bordism and symmetric squaring}

Finally in this section, we can prove that symmetric squaring induces a well-defined map in bordism. This means first determining how such a map should be defined and for this purpose we need to say a lot about neighbourhoods of the diagonals of squared manifolds. After defining the symmetric squaring in bordism we prove that it is well-defined in unoriented and oriented bordism in \autoref{uncase} and \autoref{orcase} respectively.

\subsection{Diagonal problems}\label{diagsec}
Up to this point, we have come across several instances where it was necessary to work relative to or to cut out the diagonal of a squared manifold. Those were for example
\begin{itemize}
 \item Starting with a smooth manifold $M$ and performing the symmetric squaring on it doesn't give back a smooth manifold, since there is no canonical smooth structure on the projection of the diagonal in $M\times M$ to $\sys{M}$ with respect to the coordinate interchanging map.
\item Furthermore, to define the fundamental class in \v{C}ech homology, we would like to know that to compact, smooth manifolds with boundary there can be assigned compact, smooth symmetric squared manifolds with boundary in a reasonable way.
\item From \autoref{relbord} we know that regarding \v{C}ech bordism groups relative to a neighbourhood of the diagonal means working relative to a neighbourhood of the diagonal of the occurring squared singular manifolds as long as the appearing maps from the manifolds to the topological spaces map neighbourhoods of the diagonal in the squared manifolds to neighbourhoods of the squared spaces.
\end{itemize}
These lead to requirements a neighbourhood of a diagonal of a squared manifold needs to meet so that we can use it appropriately in the following. Since we aim at dealing with bordism, we state these requirements in the context of two bordant singular manifolds and the bordism between them. So let $(M_0,\partial M_0;f_0)$ and $(M_1,\partial M_1;f_1)$ be two (unoriented) bordant singular $n$-manifolds in $(X,A)$. Furthermore, let $(W,\partial W; F)$ the bordism between them, which means
\begin{enumerate}
			\item $W$ is a compact $(n+1)$-manifold with boundary.
			\item It is $\partial W= M_0\cup M_1\cup M^{\prime}$ with $\partial M^{\prime}=\partial M_0\sqcup \partial M_1$ and $M_i\cap M^{\prime}=\partial M_i$ for $i=0, 1$.
			\item $F_{|{M_i}}=f_i$ for $i=0, 1$.
			\item $F(M^{\prime})\subset A$.
\end{enumerate}
Choose a neighbourhood $U_\Delta$ of the diagonal in $X\times X$. Then we would like to know that there exists a neighbourhood $V_\Delta$ of the diagonal in $W\times W$, satisfying the following conditions.
\begin{itemize}
	\item[\textbf D1] The complement of the neighbourhood $V_\Delta$ in $W \times W$ is a smooth compact manifold with boundary.
	\item[\textbf D2] The neighbourhood $V_\Delta$ is symmetric in the sense that for each point $(w_1,w_2)$ contained in the neighbourhood $\tau(w_1,w_2)=(w_2,w_1)$ is contained in it as well, so that the quotient of the complement by $\tau$ is again a smooth compact manifold.
	\item[\textbf D3] The neighbourhood $V_\Delta$ is mapped to $U_\Delta$ by $F \times F$.
	\item[\textbf D4] The intersection of $V_\Delta$ with the manifolds $(M_i \times M_i,\partial(M_i \times M_i),f_i\times f_i)$ for $i=0,1$ respectively are neighbourhoods of the diagonals $\Delta_{M_i}$ in $M_i \times M_i$, which satisfy conditions similar to \textbf{D1}-\textbf{D3}.
\end{itemize}
It was explained above why conditions \textbf{D1} and \textbf{D3} are required. Condition \textbf{D2} is needed, however, to be able to control the smoothness of the quotient of the squared manifold reduced by a neighbourhood of the diagonal in a sufficient way. In the proofs of \autoref{halfun} and \autoref{halfor}, a bordism is constructed inside a squared bordism just as $W\times W$ is. In the course of this construction it will be necessary to have exact knowledge about how $V_\Delta$ intersects the manifolds $M_i\times M_i$ for $i=0,1$ and this is why \textbf{D4} is needed.\\
One important property these requirements imply is that choosing two different neighbourhoods satisfying \textbf{D1}-\textbf{D4} does not give two different bordism classes. More precisely, we can prove the following remark.

\begin{rem}\label{diag-choice}
Let $(M,\partial M,f)$ be a singular $n$-manifold in $(X,A)$. Furthermore, let $({\widetilde{M \times M}})_1$ and $(\widetilde{M \times M})_2$ represent two different possibilities of choices for removing the diagonal, i.e. two different choices of neighbourhoods of the diagonal $\Delta_M$, satisfying conditions \textbf{D1}-\textbf{D4}, have been removed from $M\times M$ to result in the manifolds $({\widetilde{M \times M}})_1$ and $(\widetilde{M \times M})_2$ respectively. Then $({\widetilde{M \times M}})_1/\tau$ and $(\widetilde{M \times M})_2/\tau$ are bordant manifolds in $\mathcal{N}_{2n}\left(\sys{X},pr(X\times A\cup A \times X \cup U_{\Delta})\right)$.
\end{rem}
\begin{proof}
Implicitely, this proof uses a modification of a remark from \cite{MR548463}, which was cited before as \autoref{relbord} on page \pageref{relbord}. More precisely, we choose a third manifold $B$, use the remark to see that both $({\widetilde{M \times M}})_1/\tau$ and $(\widetilde{M \times M})_2/\tau$ are bordant to $B$ and then use associativity of the bordism relation to show the above claim.\\
First choose a neighbourhood of the diagonal contained in both given neighbourhoods. This results in choosing a third manifold $({\widetilde{M \times M}})_3$ containing both given choices. Denote $B:=({\widetilde{M \times M}})_3/\tau$.\\ A bordism between $({\widetilde{M \times M}})_1/\tau$ and $(\widetilde{M \times M})_2/\tau$ is then given by the singular $(2n+1)$-manifold $\left((B \times [0,1],\partial(B\times [0,1]); F\right)$, where $F([(m_1,m_2),(m_2,m_1),t])=[(f(m_1),f(m_2)),(f(m_2),f(m_1))]$ for all $t\in[0,1]$:
\begin{itemize}
 \item $\partial(B\times [0,1])=\partial B\times [0,1]\cup B\times 0\cup B \times 1$. The way $B$ was chosen assures that $(\widetilde{M \times M})_1/\tau$ and $(\widetilde{M \times M})_2/\tau$ can be regarded as submanifolds of $B\times 0$ and $B\times 1$ respectively. 
 \item Using this identification, the map $F$ is equal to $f$ on $(\widetilde{M \times M})_1/\tau$ and $(\widetilde{M \times M})_2/\tau$.
 \item Property \textbf{D3} makes sure that $F$ maps the complements of $({\widetilde{M \times M}})_1/\tau$ and $(\widetilde{M \times M})_2/\tau$ in $\partial(B\times [0,1])$ to $pr(X\times A\cup A \times X \cup U_{\Delta})$.
\end{itemize}
 \end{proof}

Using some of the tools provided in the preceding section, we can now prove the existence of neighbourhoods satisfying the required properties.

\begin{prop}\label{diagonal}
There exists a neighbourhood of the diagonal $\Delta_W$ in $W\times W$ satisfying \textbf{D1-D4}.
\end{prop}
\begin{proof}
If both $M_0$ and $M_1$ are the empty manifold, then it can happen that $W$ has empty boundary. In that case the following construction of the neighbourhood of the diagonal can be performed more or less in the same way. Some steps can or have to be omitted then, however, which makes it even easier. So from now on it will be assumed that $\partial W\neq\emptyset$.\\ 
The idea of how to construct the desired neighbourhood is to use a Riemannian metric to get a squared distance function which is smooth in a neighbourhood of the diagonal and then to look at inverse images of regular values of this function. Some of these steps only work in manifolds without boundary and this explains the first step in the construction.\\
For the manifold $W\times W$ with boundary, let $D(W\times W)$ be its double\footnote{Compare \autoref{subsec:double} for a definition.}. This is a manifold without boundary which we are going to examine very closely in the following. We start by explaining why the swapping of coordinates $\tau$ acts smoothly on this manifold. This is clear for all points that are surrounded by charts that are only crossed charts of the manifold $W$. But the smooth structure of $D(W\times W)$ does not only involve such charts. During its construction angles were straightened first to get a smooth manifold $W\times W$, which was then glued to a copy of itself with the help of collars. Angles were only straightened at points in $\partial W\times \partial W$.\\ As explained in \autoref{straight}, during the straightening process there are constructed charts around points in $\partial W\times \partial W$ of the form
 \begin{align}\label{chart_straight}
U\times U\rightarrow \partial W\times \partial W\times  \mathbb R\times \mathbb R_+,
\end{align} where $U$ denotes a neighbourhood of the boundary $\partial W$ in $W$.\footnote{To get charts that map to a subset of $\mathbb R^n \times \mathbb R^n \times \mathbb R\times \mathbb R_+$, one has to combine this with charts of the smooth manifold $\partial W$.} Expressed in this charts the boundary of $W\times W$ corresponds to the subset $\partial W\times \partial W\times \mathbb R\times 0 \subset \partial W\times \partial W\times  \mathbb R\times \mathbb R_+$.\\ For the smooth charts at the glueing-points of the double of $W \times W$  a collar is used, i.e. a diffeomorphism from an open neighbourhood of the boundary of $W\times W$ to $\partial (W\times W) \times \mathbb R_+$.\\
Combining the last two facts, we see that (\ref{chart_straight}) can be regarded as a collar of $W\times W$ as well.\footnote{Collars are unique, c.f. \autoref{rem:colluni}. That is why the differentiable structure of the double does not depend on the chosen collar.} This means that we can use (\ref{chart_straight}) as a chart for the double as explained in \autoref{double}.\\ The chart (\ref{chart_straight}) is constructed  with the help of the diffeomorphism\begin{align*}
\alpha\colon\mathbb R_+\times \mathbb R_+ \smallsetminus (0,0) &\rightarrow \mathbb R\times \mathbb R_+\smallsetminus (0,0) \text{ given by }\\
(r,\theta)&\mapsto(r,2\theta) \text{ for } 0\leq\theta\leq\pi/2  \text{ in polar coordinates.}
\end{align*} Using this, the map representing the action $\tau$ can be written as 
\[
(x,y,r,\theta)\mapsto (y,x,r, \pi-\theta) \text{ for } 0\leq\theta\leq\pi \text{ in polar coordinates}
\] in the neighbourhood of boundary points of $W\times W$ in $D(W\times W)$. So $\tau$ acts smoothly on $D(W\times W)$.\\
This means that we can choose a Riemannian metric on $D(W\times W)$ which is invariant under the action $\tau$.\footnote{The existence of such a metric follows from Theorem VI.2.1 in \cite{MR0413144}.} Furthermore, as it is a fixed point set of a smooth action, the double $D(\Delta_W)$ of the diagonal $\Delta_W$ in $W\times W$ is then a compact, embedded, Riemannian submanifold of $D(W\times W)$. So \autoref{8-5} on page \pageref{8-5} assures that there exists a neighbourhood of $D(\Delta_W)$ in $D(W\times W)$ such that the square of the distance function $R^2\colon D(W\times W)\rightarrow \mathbb R_+$, which assigns to every point the square of its distance from the double-diagonal $D(\Delta_W)$, is smooth in that neighbourhood. \\
Because of Sard's Theorem\footnote{See for example Chapter 2 in \cite{MR0226651}.}, we can choose a regular value $\delta>0$ of $R^2$ and $R^2_{|\partial(W\times W)}$ such that
\begin{equation} \label{D3}
 (R^2)^{-1}(\left[0,\delta\right))\subset (F\times F)^{-1}(U_\Delta).
\end{equation}
The next and last step of the construction is now to undo the doubling of the manifold $W\times W$, i.e. the claim is that $V_\Delta:=(R^2)^{-1}\left(\left[0,\delta\right)\right)\cap W\times W$ is a neighbourhood of the diagonal with the properties \textbf{D1-D4}. The way $\delta$ was chosen in (\ref{D3}) immediately tells us that \textbf{D3} is satisfied by $V_\Delta$.\\ Condition \textbf{D2} is fulfilled since we defined $V_\Delta$ by means of the distance from the diagonal and because we chose a $\tau$-invariant Riemannian metric to determine that distance. From these two facts it follows that points $(x,y)$ have the same distance from the diagonal as points $(y,x)$ have.\\ To prove \textbf{D1}, we have to take a close look at the charts involved in the construction, because we have to show that the complement of $V_\Delta$, i.e. $\mathcal V:=(W\times W)\cap (R^2)^{-1}\left(\left[\delta,\infty\right)\right)$, is a smooth $2n+2$-dimensional manifold with boundary. For $(x,y)\in\mathcal V$, we distinguish between the following cases:
\begin{enumerate}[(a)]
 \item $(x,y)\notin \partial(W\times W)$ and $R^2(x,y)>\delta$.\\ This is the easiest case, concerning only the interior of $\mathcal V$. In a neighbourhood of this point the original charts of $W\times W$ can be used for a manifold structure on $\mathcal V$, since these are used for the differentiable structures in the doubling-construction as well as they can be used to equip the inverse image set $(R^2)^{-1}\left(\left(\delta,\infty\right)\right)$ with a differentiable structure.
\end{enumerate} 
The following three cases concern the boundary of $\mathcal V.$
\begin{enumerate}[(a)]
\stepcounter{enumi}
 \item $(x,y)\notin \partial(W\times W)$ and $R^2(x,y)=\delta$.\\ For points in $(R^2)^{-1}(\delta)$ boundary charts fitting together with the interior ones are constructed as follows. For simplicity define $h:=(R^2-\delta)\colon D(W\times W)\rightarrow \mathbb R$, such that zero is a regular value for $h$.  Using the properties of that regular value $0$ of $h$, find charts\footnote{See for instance ``Local Submersion Theorem'' on page 20 in \cite{MR0348781} to learn how.} $\varphi\colon U\rightarrow \mathbb R^{2(n+1)}$ of D($W\times W)$ around $(x,y)$ and a sign-preserving chart $\psi\colon V\rightarrow \mathbb R$ of $\mathbb R$ around $0$, such that $\tilde{h}:=(\psi\circ h\circ\varphi^{-1})\colon\varphi(U)\rightarrow \mathbb R$ is of the form $(x_1,x_2,\ldots, x_{2(n+1)})\mapsto x_1.$\\
Then $\tilde{h}^{-1}(\left[0,\infty)\right)=\left\{(x_1,x_2,\ldots,x_{2(n+1)})\in\varphi(U)\colon x_1\geq 0\right\}$ is open in $\mathbb R_+^{2(n+1)}$ and $\varphi$ restricted to $\varphi^{-1}(\tilde{h}^{-1}(\left[0,\infty)\right))$ is a boundary chart for $(R^2)^{-1}\left(\left[\delta,\infty\right)\right)$ around $(x,y)$. Since $\varphi$ is a chart of $W\times W$ around $(x,y)$ as well, because $(x,y)$ does not lie in the boundary of $W\times W$, this construction gives smooth boundary charts for $\mathcal V$. \label{delta}
 \item $(x,y)\in\partial(W\times W)$ and $R^2(x,y)>\delta$:\\Since we deal with a point in the boundary $\partial(W\times W)$ and the smooth structure that it inherits as a subset of the double $ D(W\times W)$, we have to think about the collar in (\ref{chart_straight}) above in connection with the map $R^2$. In terms of this collar the map $R^2$ can be written as\footnote{For this to work we assume that the Riemannian metric we chose at the beginning was chosen in a way such that it gives the true distances in the above situation.}
\begin{align}\label{R^2}
(x,y,r,\theta)\mapsto r^2cos^2\theta + \frac{1}{2}d(x,y)^2 \text{ for }0\leq\theta \leq \pi \text{ in polar coordinates,} 
\end{align}since the diagonal corresponds to the $y$-axis in the half plane $\mathbb R\times\mathbb R_+$ above. Here $d(-,-)$ denotes the Riemannian distance between two points and to see the above formula we use that the squared distance in product spaces is the sum of the squared distances in each single space. This is a smooth map and the inverse image of the set $(\delta,\infty)$ with respect to this map is open in $\partial W\times \partial W\times \mathbb R\times \mathbb R_+$ and has a smooth structure coming from the collaring charts described above. 
\label{collar} 
 \item $(x,y)\in\partial(W\times W)$ and $R^2(x,y)=\delta$.\\ In this case it will be necessary to straighten angles as in \autoref{straight} on page \pageref{straight}. To be able to straighten angles of $\mathcal V$, this lemma requires the subset $\tilde{V}:=\left\{(x,y)\in \mathcal V |(x,y)\in\partial(W\times W)\text{ and }R^2(x,y)=\delta\right\}$ to be a $2n$-dimensional differentiable submanifold of $\mathcal V$. It also requires $\tilde{V}$ to have a neighbourhood $U$ and a homeomorphism $\phi$ onto $\tilde{V}\times \mathbb R_+ \times \mathbb R_+$ with $\phi(x)=(x,0,0)$ for $x\in\tilde{V}$ and where $\phi$ is a diffeomorphism from $U\smallsetminus \tilde{V}$ onto $\tilde{V}\times \mathbb R_+ \times \mathbb R_+ \smallsetminus\tilde{V}\times 0\times 0$.\\ To see this, we combine the two preceding cases. First look at the smooth map (\ref{R^2}) and its restriction to the boundary of $W\times W$. Using Sard's Theorem again, we can assume, that $\delta$ is a regular value for both these maps expressed in terms of the collaring charts. The inverse image of $\delta$ with respect to the restricted map is then a smooth manifold and it is also exactly the subset $\tilde{V}$ required to be a submanifold above.\\ For the other requirement, recall that in (\ref{collar}) it was shown that the inverse image of $(\delta,\infty)$ can be viewed as a subset of $\partial W\times \partial W\times \mathbb R\times \mathbb R_+$. Using the same methods as in (\ref{delta}) with respect to this subset, one can show that the inverse image of $\left[\delta,\infty\right)$ can be viewed as a subset of $\partial W\times \partial W\times \mathbb R_+\times \mathbb R_+$. But this is a neighbourhood of $\tilde{V}$ and the maps just used while performing the construction as in (\ref{delta}) and (\ref{collar}) can be combined to give a map with the properties requested for $\phi$.  This means that we can straighten angles here as in \autoref{straight} and so there exists a smooth structure on $\mathcal{V}$.
\end{enumerate}
What is left to prove for \textbf{D4} is that $M_i\times M_i\cap(R^2)^{-1}\left(\left[\delta,\infty\right)\right)$ are smooth compact manifolds with boundary for $i=0,1$, since the other two properties are satisfied automatically. Without loss of generality, using Sard's Theorem again, we can assume that $\delta$ is a regular value for the restriction of $R^2$ to the smooth compact manifolds $M_i \times M_i$  and their boundaries for $i=0,1$ respectively. This means that we are in the same situation as in the above investigation. Namely, we have to deal with a smooth manifold with boundary of the form of a cartesian product where a neighbourhood of the diagonal is cut out with the help of a squared distance function. That is why we can repeat similar arguments as in the previous reasoning to show that the complement of this neighbourhood in the manifolds $M_i \times M_i$ is itself a smooth manifold with boundary.
\end{proof}

\begin{note}\label{schlange}
From now on, a cartesian product $\textbf{-} \times \textbf{-}$ where a neighbourhood of the diagonal is removed satisfying all \textbf{D1}-\textbf{D4} will be indicated by $\widetilde{\textbf{-} \times \textbf{-}}$, more precisely $\widetilde{\textbf{-} \times \textbf{-}}:=(\textbf{-} \times \textbf{-})\smallsetminus(V_\Delta\cap(\textbf{-} \times \textbf{-}))$, where $V_\Delta$ denotes a neighbourhood of the diagonal as defined in the proof above.
\end{note}

\subsection{The unoriented case}\label{uncase}

It is now time to prove one of our main results, namely that symmetric squaring induces a well-defined map in (unoriented) bordism.
\begin{thm}\label{halfun}
 Let $(X,A)$ be a pair of topological spaces. The map \[
\sys{(\hspace{0,5em}\cdot\hspace{0,5em})}\colon\mathcal{N}_n(X,A)\rightarrow\check{\mathcal{N}}_{2n}\left(\sys{(X,A)}\right)
\] defined by mapping $(M,\partial M;f)$ to \[\sys{(M,\partial M;f)}:=\left\{\widetilde{(M \times M)}/\tau,\partial(\widetilde{(M \times M)}/\tau);\sys{f}_{|{\widetilde{(M \times M)}/\tau}} \right\}_{U_i\supset \Delta_X \text{open}},
\]
where each $\widetilde{M \times M}$ is chosen such that it satisfies conditions \textbf{D1}-\textbf{D4} with respect to one open neighbourhood $U_i$ of $\Delta_X$, is well defined.
\end{thm}

\begin{proof}
 Let $(M_0,\partial M_0;f_0)$ and $(M_1,\partial M_1;f_1)$ be two bordant singular $n$-manifolds. Furthermore, let $(W,\partial W; F)$ be a bordism between them, which means
\begin{enumerate}
			\item $W$ is a compact $(n+1)$-manifold with boundary.
			\item It is $\partial W= M_0\cup M_1\cup M^{\prime}$ with $\partial M^{\prime}=\partial M_0\sqcup \partial M_1$ and $M_i\cap M^{\prime}=\partial M_i$ for $i=0, 1$.
			\item $F_{|{M_i}}=f_i$ for $i=0, 1$.
			\item $F(M^{\prime})\subset A$.
\end{enumerate}
To show that the map in question is well-defined, we need to show that bordant elements are mapped to bordant elements. To be bordant in $\check{\mathcal{N}}_{2n}\left(\sys{(X,A)}\right)$ means to be bordant in each level of the inverse limit, i.e. in each bordism group of the limit related to one neighbourhood of the diagonal in $X\times X$. So we fix one open neighbourhood $U_{\Delta}$ of the diagonal $\Delta_X$ in $X\times X$ and look at $\widetilde{M_i \times M_i}$ with respect to this neighbourhood $(i=0, 1)$.\\
A singular bordism between $(\widetilde{(M_0 \times M_0)}/\tau,\partial(\widetilde{(M_0 \times M_0)}/\tau);\sys{{f_0}}_{|\widetilde{{(M_0 \times M_0)}/\tau}})$ and \linebreak $(\widetilde{(M_1 \times M_1)}/\tau,\partial(\widetilde{(M_1 \times M_1)}/\tau);\sys{{f_1}}_{|\widetilde{({M_1 \times M_1})/\tau}})$ is now to be constructed, valid in the group $\mathcal{N}_{2n}(\sys{X},pr(X\times A\cup A \times X \cup U_{\Delta}))$. This will be done in several steps and it will be done using the above given bordism $(W,\partial W;F)$. We start by constructing a bordism between the squared manifolds disregarding the action of the involution $\tau$. For this we use a Morse function and look at the fibred product with respect to this function as a subset of $W\times W$. Then we analyse the constructed bordism and observe that dividing it by the involution $\tau$ gives the desired bordism between the symmetric squared manifolds.\\
So we first choose a Morse function $f\colon W\rightarrow\left[0,1\right]$, as in \autoref{meinmorse}, with the following properties
\begin{enumerate}
 \item $f$ does not have any critical points in a neighbourhood of $\partial W$,
 \item $f^{-1}(i)=M_i$ for $i=0,1$,
 \item $f_{|M^\prime}$ is a Morse function on the triad $(M^\prime;\partial M_0,\partial M_1)$,
 \item $f$ takes distinct values at distinct critical points of $f$ and additionally no critical point of $f_{|M^\prime}$ has the same value as any critical point of $f$.
\end{enumerate}
Now we form the fibred product $B$ with respect to this map, i.e.  
\[
B=\left\{(x,y)\in W\times W|f(x)=f(y)\right\}.
\]This can as well be interpreted as an inverse image of the map 
\begin{align*}\bar{f}\colon W\times W&\rightarrow\left[-1,1\right],\text{ defined by }\\
(x,y)&\mapsto f(x)-f(y), \text{ namely } B=\bar{f}^{-1}(0).                                                                  
\end{align*}
 Consider now first the (non-compact) set $B\smallsetminus \Delta_W$. It will be shown that this consists only of regular points for $\bar{f}$ and $\bar{f}_{|\partial(W\times W)}$, so\footnote{See Proposition VIII.5.2 in \cite{3110162369} or the Preimage Theorem in section 1.4 of \cite{MR0348781}. This theorem is sometimes called Regular Value Theorem or Submersion Theorem as well. It will be used in various versions in the following.} $B\smallsetminus \Delta_W$ has the structure of a smooth (non-compact) $(2n+1)$-manifold with boundary and its boundary\footnote{To avoid complicated notation, we speak of the boundary of $B$ and write $\partial B$ although $B$ is not a manifold everywhere.} is $\partial B=B\cap\partial(W \times W)$.\\
Regularity of the map $\bar{f}$ is given at a point $(x,y)$ if and only if at least one of the points $x$ or $y$ in $W$ is a regular point for $f$. For $x\neq y$, which are both critical for f, it follows from the fourth condition above, that $f(x)\neq f(y)$, so $(x,y)\notin B$, which shows that $B\smallsetminus\Delta_W$ only consists of regular points for $\bar{f}$.\\ To obtain the fact that the boundary of $B$ is just the intersection of the boundary of $W\times W$ with $B$, we need to show that the restriction $\bar{f}_{|\partial{(W\times W)}}$ to the boundary of $W\times W$ is regular on $(B\cap\partial(W \times W))\smallsetminus \Delta_W$. So let now $x\neq y$ with $(x,y)\in\partial(W\times W)$ and without loss of generality let $x\in\partial W$. The first assumption for $f$ then tells us that $x$ is regular for $f$. Thus, for $(x,y)$ being a candidate for a critical point for $\bar{f}_{|\partial(W\times W)}$ it is necessary that $x$ is critical for $f_{|\partial W}$ and $y$ is critical for $f$. It would follow from the first property of $f$ again that $y\notin \partial W$, if (x,y) was critical in this context. The second and fourth condition above together then would yield $f(x)\neq f(y)$. This means that $(x,y)\notin B$ if $(x,y)$ was critical.\\
As in the proof for the existence of a neighbourhood of the diagonal\footnote{See \autoref{diagonal}.} that has a complement which is a smooth compact manifold with boundary, it can thus be shown that $B$ reduced by the above chosen neighbourhood of the diagonal $\Delta_W$ in $W\times W$ is a smooth compact manifold with boundary of dimension $2n+1$. To use this proposition here, you have to choose a regular value for the (locally smooth) function that represents the squared distance from the diagonal on $W \times W$ and $B$ simultaneously, which is possible because of Sard's Theorem.\footnote{See for example Chapter 2 in \cite{MR0226651}.}\\From now on let $\tilde{B}$ denote this manifold, which in fact is even more than that, it is a bordism.
\begin{claim}\footnote{The way we chose the neighbourhood of the diagonal depending on $U_\Delta$ does not change the bordism class of  $(\tilde{B}, \partial \tilde{B}; (F \times F)_{|\tilde{B}})$, compare \autoref{diag-choice}.}
The singular $2n-$manifold $\left(\widetilde{M_0 \times M_0},\partial(\widetilde{M_0 \times M_0});(f_0 \times f_0)_{|\widetilde{M_0 \times M_0}}\right)$ and the singular $2n-$manifold $\left(\widetilde{M_1 \times M_1},\partial(\widetilde{M_1 \times M_1});(f_1 \times f_1)_{|\widetilde{M_1 \times M_1}}\right)$ are bordant in $\mathcal{N}_{2n}(X\times X, X\times A\cup A \times X\cup U_{\Delta})$ via the singular manifold $(\tilde{B}, \partial \tilde{B}; (F \times F)_{|\tilde{B}})$.
\end{claim}
 For this claim it is left to be checked that
\begin{enumerate}
\renewcommand{\labelenumi}{(\alph{enumi})}
	\item $\partial \tilde{B}=\widetilde{(M_0 \times M_0)}\,\cup\, \widetilde{(M_1 \times M_1)}\,\cup\, \tilde{M}$ and $\partial \tilde{M}=\partial\left(\widetilde{M_0 \times M_0} \right)\,\sqcup\, \partial\left(\widetilde{M_1 \times M_1} \right)$ and $\left(\widetilde{M_i \times M_i} \right)\,\cap \tilde{M}=\partial\left(\widetilde{M_i \times M_i} \right)$ for $i=0, 1$.
	\item $(F \times F)_{|\widetilde{(M_i \times M_i)}}=(f_i \times f_i)_{|{\widetilde{(M_i \times M_i)}}}$ for $i=0, 1$.
	\item $(F \times F)(\tilde{M})\subset X\times A\,\cup\,A \times X\,\cup\,U_{\Delta}$.
\end{enumerate}
To prove this claim, we go on in two steps. Since the diagonal of $W\times W$ can be cut out in a way fitting nicely with submanifolds of $W\times W$ as is shown in \autoref{diagonal}, it is easier and possible to analyse all appearing submanifolds including the neighbourhood of the diagonal as a first step. Secondly, we will then look at what happens to the accomplished results as soon as the neighbourhood of the diagonal is cut out. We will see that because of the careful way we chose the neighbourhood of the diagonal, everything will work out fine after cutting the diagonal out.\\
To show a version of (a) including the neighbourhood of the diagonal first, we use the result about the boundary of $B$ that we have shown above and start computing.
	\begin{align}\label{rand}
\begin{split}
	\partial B =& \bar{f}^{-1}(0)\cap\partial(W \times W)\\
	                =& \bar{f}^{-1}(0)\cap\left(\partial W \times W\cup W \times \partial W\right)\\
	                =& \bar{f}^{-1}(0)\cap\Big(M_0 \times M_0\cup M_1 \times M_1\cup M^{\prime} \times                    W\cup W \times M^{\prime}\\
	                 &\bigcup\limits_{i=0, 1}\left(M_i \times \left(W\smallsetminus M_i\right)\cup \left(W\smallsetminus M_i\right) \times M_i\right)\Big)	\\
	                =&M_0 \times M_0\cup M_1 \times M_1\cup\left(\bar{f}^{-1}(0)\cap\left(M^{\prime} \times                             W\cup W \times M^{\prime}\right)\right), \end{split}
  	\end{align}
		since from $f^{-1}(i)=M_i$ for $i=0,1$, it follows that $M_0 \times M_0\cup M_1 \times M_1\subset \bar{f}^{-1}(0)$ and $\bar{f}^{-1}(0)\cap\Big(\bigcup\limits_{i=0, 1}\left(M_i \times \left(W\smallsetminus M_i\right)\right)\Big) =\emptyset$ for $i=0,1$.\\
		The properties of $\bar{M}:=\left(\bar{f}^{-1}(0)\cap\left(M^{\prime} \times W\cup W \times M^{\prime}\right)\right)$ need to be investigated next:
	\begin{align}\label{eq:restrandschnitt}\notag
   	\bar{M}\cap M_i \times M_i&=M_i \times M_i\cap(W \times M^{\prime}\cup M^{\prime} \times W) &&     \text{ by } M_i \times M_i\subset \bar{f}^{-1}(0)\\ \notag
   	 &= M_i \times \partial M_i\cup \partial M_i \times M_i &&\text{ by } M_i\cap M^{\prime}=\partial M_i  \\    & = \partial\left(M_i \times M_i\right) && 
   	\end{align}
   	and that is one of the properties of $\bar{M}$ which we have to show. For the other one, the boundary of $\bar{M}$ has to be examined. To make it easier to look at $\partial\bar{M}$, we remark the following equalities that all follow from the fact $M^{\prime}\subset W$ or from $M^{\prime}\subset \partial W$:
\begin{align*}
    M^{\prime} \times M^{\prime}&=W \times M^{\prime}\cap M^{\prime} \times W\\
   \partial(M^{\prime} \times W)\cap\partial(W \times M^{\prime})&=(\partial M^{\prime} \times W \cup M^{\prime} \times \partial W)\cap (\partial W \times M^{\prime}\cup W \times \partial M^{\prime})\\
  &=\partial M^{\prime}\times M^{\prime}\cup\partial M^{\prime}\times \partial M^{\prime}\cup M^{\prime}\times M^{\prime}\cup M^{\prime}\times\partial M^{\prime}\\
&=M^{\prime}\times M^{\prime}
\end{align*}
We therefore altogether have that 
\begin{align*}
  M^{\prime} \times M^{\prime}&=W \times M^{\prime}\cap M^{\prime} \times W&&\\
    &=\partial(M^{\prime} \times W)\cap\partial(W \times M^{\prime}).&&
  \end{align*}
We can assume\footnote{This is possible since everything takes place inside the manifold $W \times W$  and the submanifolds are sufficiently regular.} here that the boundary of the union of manifolds with boundary, which intersect only in the intersection of their boundaries, is the union of their boundaries with the (interiors of\footnote{We use $(-)^{\circ}$ to denote the interior, more precisely the notation means that what is between the brackets is to be regarded without its boundary.}) the intersections removed. This together with the above is what can be used to compute $\partial(W \times M^{\prime}\cup M^{\prime} \times W)$:
\begin{align*}
\begin{split}
  &\partial(W \times M^{\prime}\cup M^{\prime} \times W)\\  
  =&\left(\partial(W \times M^{\prime})\cup\partial(M^{\prime} \times W)\right)\smallsetminus\left(M^{\prime} \times M^{\prime}\right)^{\circ}\\
  =&\left(\left(\partial(W \times M^{\prime})\cup\partial(M^{\prime} \times W)\right)\smallsetminus\left(M^{\prime} \times M^{\prime}\right)\right)\cup \partial(M^\prime\times M^\prime),
\end{split}
\end{align*}
because $\partial(M^\prime\times M^\prime)\subset \partial(W \times M^{\prime})\cup\partial(M^{\prime} \times W)$ and therefore
\begin{align}\label{rand2}
\begin{split}
&\partial(W \times M^{\prime}\cup M^{\prime} \times W)\\
  =&\left((\partial W\times M^{\prime}\cup M^{\prime}\times\partial W\cup W\times\partial M^{\prime}\cup \partial M^{\prime}\times W)\smallsetminus\left(M^{\prime} \times M^{\prime}\right)\right)\cup \partial(M^\prime\times M^\prime) \\
  =&M_0^\circ\times M^{\prime}\cup M_1^\circ\times M^{\prime}\cup M^{\prime}\times M_0^\circ \cup M^{\prime}\times M_1^\circ\cup  (W\smallsetminus M^{\prime})\times \partial M_0 \cup\\& (W\smallsetminus M^{\prime})\times \partial M_1
\cup\partial M_0\times (W\smallsetminus M^{\prime})\cup \partial M_1\times (W\smallsetminus M^{\prime})\cup \partial(M^\prime\times M^\prime)\\
  =&M_0\times M^{\prime}\cup M_1\times M^{\prime}\cup M^{\prime}\times M_0 \cup M^{\prime}\times M_1\cup  (W\smallsetminus M^{\prime})\times \partial M_0 \cup\\& (W\smallsetminus M^{\prime})\times \partial M_1
\cup\partial M_0\times (W\smallsetminus M^{\prime})\cup \partial M_1\times (W\smallsetminus M^{\prime})
\end{split}
  \end{align}
where the identities $\partial W= M_0\cup M_1\cup M^{\prime}$ and $\partial M^{\prime}=\partial M_0\sqcup \partial M_1$ and $M_i\cap M^{\prime}=\partial M_i$ for $i=0, 1$  are used. Now that we know how the boundary of $W \times M^{\prime}\cup M^{\prime} \times W$ looks like, a similar analysis as before can be made to conclude that $\bar{f}_{|W \times M^{\prime}\cup M^{\prime} \times W}$ and $\bar{f}_{|\partial(W \times M^{\prime}\cup M^{\prime} \times W)}$ are regular on $B$. From this it follows that
\begin{align}\label{restrand}
\begin{split}
\partial\bar{M}&=\partial(B\cap(W \times M^{\prime}\cup M^{\prime} \times W))\\
&=B\cap \partial(W \times M^{\prime}\cup M^{\prime} \times W)\\
&=M_0\times\partial M_0\cup M_1\times\partial M_1\cup\partial M_0\times M_0\cup \partial M_1\times M_1.
\end{split}
\end{align}
The last step is deduced from (\ref{rand2}) together with the identities $f^{-1}(i)=M_i$ and $M_i\cap M^{\prime}=\partial M_i$ for $i=0, 1$.\\
As announced earlier, we will now look at these results again and try to explain what happens in the equations as soon as we cut out a neighbourhood of the diagonal. For this purpose, a notation is needed to indicate the resulting new boundary parts that arise from removing the neighbourhood of  the diagonal from $W\times W$ in the way described in \autoref{diagsec}. Using the notation from the proof of \autoref{diagonal} on page \pageref{diagonal}, let $N:=(R^2)^{-1}(\delta)\cap W\times W$, which gives
\[
\partial(\widetilde{W\times W})=\widetilde{W\times\partial W}\cup\widetilde{\partial W\times W}\cup N \text{ with } \partial N=\partial(W\times W)\cap N.
\] The way the neighbourhood of the diagonal was chosen, namely property \textbf{D4}, now allows the conclusion that for submanifolds of $W\times W$ their new boundary is always $N$ intersected with the submanifold. Therefore we deduce from the computations above
\begin{align*}
\begin{split}
 \partial\tilde{B}=&\widetilde{M_0\times M_0}\cup\widetilde{M_1\times M_1}\cup(B\cap(\widetilde{M^{\prime}\times W} \cup\widetilde{W\times M^{\prime}}))\cup(B\cap N)\text{ so }\\
\tilde{M}:=&(B\cap(\widetilde{M^{\prime}\times W}\cup\widetilde{W\times M^{\prime}}))\cup(B\cap N)\text{ which gives }\\
\partial(\widetilde{M_i\times M_i}) =&\tilde{M}\cap(\widetilde{M_i\times M_i}) \text{ using (\ref{eq:restrandschnitt}). }\\ 
\end{split}
\end{align*}
This already shows two thirds of statement (a) in the claim we are about to prove. The same reasoning turns the corresponding equation (\ref{restrand}) into 
\begin{multline*}
\partial(B\cap(\widetilde{W\times M^{\prime}}\cup\widetilde{M^{\prime}\times W}))=\\
\widetilde{M_0\times\partial M_0}\cup\widetilde{ M_1\times\partial M_1}\cup\widetilde{\partial M_0\times M_0}\cup\widetilde{ \partial M_1\times M_1}\cup(B\cap N\cap (W\times M^{\prime}\cup M^{\prime}\times W)).
\end{multline*}
This can be used to compute $\partial\tilde M$, but more information is needed. Namely, observe that $B\cap N$ denotes the new boundary of $B$ and the way we chose $N$ means that the boundary of the new boundary of $B$ is equal to the new boundary of the boundary of $B$. In symbols $\partial(B\cap N)=\partial B\cap N$, so from (\ref{rand}) we get:
\[
\partial(B\cap N)=
\bigcup_{i=0,1}(N\cap M_i\times M_i)\cup (B\cap N\cap (W\times M^{\prime}\cup M^{\prime}\times W)). 
\]
Again using the nice behaviour of the boundary of the union of manifolds only intersecting in their boundaries we finally conclude
\begin{align*}
\begin{split}
\partial\tilde{M} =& \left(\partial\left(B\cap\left(\widetilde{W\times M^{\prime}}\cup\widetilde{M^{\prime}\times W}\right)\right)\cup\partial\left(B\cap N\right)\right)\\
&\smallsetminus \left(B\cap N\cap \left(W\times M^{\prime}\cup M^{\prime}\times W\right)\right)\\
=& \bigcup_{i=0,1}(N\cap M_i\times M_i)\cup\widetilde{M_0\times\partial M_0}\cup\widetilde{ M_1\times\partial M_1}\cup\widetilde{\partial M_0\times M_0}\cup\widetilde{ \partial M_1\times M_1} \\
=&\partial(\widetilde{M_0\times M_0})\cup\partial(\widetilde{M_1\times M_1})
\end{split}
\end{align*} which implies (a) by the special way $N$ was chosen.\\
That the restriction of $F \times F$ to $\widetilde{M_i \times M_i}$ is equal to $f_i \times f_i$ for $i=0, 1$ follows directly from the corresponding property of $(W,\partial W;F)$ as a bordism between $(M_0,\partial M_0;f_0)$ and $(M_1,\partial M_1;f_1)$, which yields (b).\\
For (c) it has to be shown that $\tilde{M}=\left(B\cap\left(\widetilde{M^{\prime} \times W}\cup\widetilde{W \times M^{\prime}}\cup N\right)\right)$ is mapped to $X\times A\cup A \times X\cup U_{\Delta}$ under $(F \times F)_{|\tilde{B}}.$ That $N$ is mapped to $U_{\Delta}$ is assured by \textbf{D3} and because of the property $F(M^{\prime})\subset A$ of the given bordism $(W,\partial W;F)$, the desired statement is true.\\
There is actually an easier bordism between the two manifolds above given by the fact that for two bordant manifolds $P$ and $Q$ the cartesian products $P \times P$ and $Q \times Q$ are bordant as well. It is given in the following way: Let $V$ denote a bordism between the two manifolds $P$ and $Q$. Then $P\times P$ is bordant to $Q\times P$ via $V\times P$, whereas $Q\times P$ is bordant to $Q\times Q$ via $Q\times V$, so gluing together these two bordisms along $Q\times P$ gives a bordism as claimed.\footnote{This works in the oriented case as well.} The problem with this easier bordism is that there is no a priori way to build the quotient of it by $\tau$.\\
However, $\tilde{B}$ was chosen in a symmetric way, i.e. such that the involution $\tau$ maps $\tilde{B}$ to itself. This can be seen firstly as $B$ was chosen as the subset of all $(x,y)$ in $W\times W$ satisfying $f(x)=f(y)$ for a chosen Morse function. So as soon as $(x,y)$ is contained in $B$, the point $(y,x)$ is as well, which makes $B$ symmetric. Additionally, the neighbourhood of the diagonal that is cut out from $B$ had to be chosen very carefully using \autoref{diagonal} as was noted earlier. This makes it possible to build the quotient of $\tilde{B}$ by $\tau$ in a sensible way, which is the main reason why the constructed bordism works as claimed:
\begin{claim}
The bordism $(\bar{B},\partial \bar{B};\sys{F}_{|\bar{B}})$ where $\bar{B}:=\tilde{B}/\tau$ is a bordism as required.
\end{claim}
There are four points to be checked. They are
\renewcommand{\labelenumi}{(\alph{enumi}')}
\begin{enumerate}
	\item $\bar{B}$ is a compact $(2n+1)$-manifold with boundary.
	\item $\partial \bar{B}=\widetilde{(M_0 \times M_0)}/\tau\,\cup\, \widetilde{(M_1 \times M_1)}/\tau\,\cup\, M$ with $\partial M=\partial\left(\widetilde{(M_0 \times M_0)}/\tau\right)\,\sqcup\, \partial\left(\widetilde{(M_1 \times M_1)}/\tau\right)$ and $\left(\widetilde{(M_i \times M_i)}/\tau\right)\,\cap M=\partial\left(\widetilde{(M_i \times M_i)}/\tau\right)$ for $i=0, 1$.
	\item $\sys{F}_{|{\widetilde{(M_i \times M_i)}/\tau}}=\sys{{f_i}}_{|{\widetilde{(M_i \times M_i)}/\tau}}$ for $i=0, 1$.
	\item $\sys{F}(M)\subset pr(X\times A)\,\cup\, pr(U_{\Delta})$
\end{enumerate}
\begin{itemize}
	\item[ad (a'):] We already know, that $\tilde{B}$ is a manifold and the action of $\tau$ on this manifold is properly discontinuous because $\tilde{B}$ does not contain any points of the form $(x,x)$, which means that $\tau$ has no fixed points. Therefore $\bar{B}$ is a smooth compact $(2n+1)$-manifold with boundary. \footnote{See Theorem 1.88 in \cite{MR2088027}.}
	\item[ad (b'):] It is $\partial \bar{B}=\partial (\tilde{B}/\tau)=\partial \tilde{B}/\tau$, since the charts of $\bar{B}$ arise from the charts of $\tilde{B}$ by carefully composing them with the projection $pr\colon\tilde{B}\rightarrow \bar{B}$ . So (b') follows from (b) by defining $M:=\tilde{M}/\tau$.
	\item[ad (c'):] This follows from (c) using the definition of $\sys{F}$ and $\sys{{f_i}}$ for $i=0, 1$.
	\item[ad (d'):] With $M$ defined as above, this follows from (d).
\end{itemize}
\end{proof}
Furthermore, we can prove that symmetric squaring is natural in the following sense.
\begin{lem}[Naturality of the construction]\label{unnat}
Let $n\in \mathbb N$, let $(X,A)$ and $(Y,B)$ be pairs of topological spaces and let $g\colon (Y,B)\rightarrow (X,A)$ be a continuous map. Then the following diagram is commutative
\[
\begin{xy}\xymatrixcolsep{3.5pc}
\xymatrix{
\mathcal{N}_n(Y,B) \ar[r]^{\sys{\cdot}} \ar[d]_{\mathcal{N}(g)} & \check{\mathcal{N}}_{2n}(\sys{(Y,B)})\ar[d]_{\check{\mathcal{N}}_{2n}(\sys{g})} \\
\mathcal{N}_n(X,A) \ar[r]_{\sys{\cdot}} & \check{\mathcal{N}}_{2n}(\sys{(X,A)})
}
\end{xy}
\]

\end{lem}
\begin{proof}
 For an arbitrary singular manifold $\left[M,\partial M, f\right]\in \mathcal{N}_n(Y,B)$, it has to be shown that
\[
\sys{\left(\mathcal{N}(g)\left[M,\partial M, f\right]\right)}=\check{\mathcal{N}}_{2n}(\sys{g})\left(\left\sys{[M,\partial M, f\right]}\right).
\]
Per definition, 
\begin{align*}
 \sys{\left(\mathcal{N}(g)\left[M,\partial M, f\right]\right)}&=\sys{\left[M, \partial M, g\circ f\right]}\\
&=\left\{\widetilde{(M \times M)}/\tau,\partial(\widetilde{(M \times M)}/\tau);\sys{(g\circ f)}_{|{\widetilde{(M \times M)}/\tau}} \right\}_{U_i\supset \Delta_X \text{open}}
\end{align*} and

\begin{align*}
 \check{\mathcal{N}}_{2n}(\sys{g})\left(\left\sys{[M,\partial M, f\right]}\right)&= 
\check{\mathcal{N}}_{2n}(\sys{g})\left\{\widetilde{(M \times M)}/\tau,\partial(\widetilde{(M \times M)}/\tau);\sys{f}_{|{\widetilde{(M \times M)}/\tau}} \right\}_{U_i\supset \Delta_X}\\
&=\left\{\widetilde{(M \times M)}/\tau,\partial(\widetilde{(M \times M)}/\tau);\sys{g}\circ\sys{f}_{|{\widetilde{(M \times M)}/\tau}} \right\}_{U_i\supset \Delta_X\text{open }}.
\end{align*}
Since the maps
\[
\sys{(g\circ f)}\text{ and }\sys{g}\circ \sys{f} \colon \widetilde{(M \times M)}/\tau,\partial(\widetilde{(M \times M)}/\tau)\rightarrow \sys{(X,A)}
\] are equal, the lemma follows.
\end{proof}

%% file: oriented.tex
\subsection{The oriented case}\label{orcase}

%

The aim of this section is to show that an analogous statement to the main \autoref{halfun} holds in the case of oriented bordism. For this to be true, we have to restrict attention to even dimensions since the involution $\tau(x,y)=(y,x)$ applied to the product $M\times M$ of an oriented manifold $M$ is orientation preserving if the dimension of $M$ is even and orientation inverting if the dimension of $M$ is odd.\footnote{See Proposition 2.8 in \cite{diplom}.}. So only in the even cases $\widetilde{(M\times M)}/\tau$ is again a well-defined manifold with orientation.

\begin{thm}\label{halfor}
 Let $(X,A)$ be a pair of topological spaces and let $n\in\mathbb N$ be even. The map \[
\sys{(\hspace{0,5em}\cdot\hspace{0,5em})}\;\;\colon\Omega_n(X,A)\rightarrow\check{\Omega}_{2n}\left(\sys{(X,A)}\right)
\] defined by mapping the oriented singular manifold $(M,\partial M;f)$ to \[\sys{(M,\partial M;f)}:=\left\{\widetilde{(M \times M)}/\tau,\partial(\widetilde{(M \times M)}/\tau);\sys{f}_{|{\widetilde{(M \times M)}/\tau}} \right\}_{U_i\supset \Delta_X \text{open}},
\]
where each $\widetilde{M \times M}$ is chosen such that it satisfies conditions \textbf{D1}-\textbf{D4} with respect to one open neighbourhood $U_i$ of $\Delta_X$, is well defined.
\end{thm}

\begin{proof}
 Let $(M_0,\partial M_0;f_0)$ and $(M_1,\partial M_1;f_1)$ be two bordant oriented singular $n$-manifolds. Furthermore, let $(W,\partial W; F)$ be a bordism between them, which means
\begin{enumerate}
			\item $W$ is a compact $(n+1)$-manifold with boundary.
			\item It is $\partial W= -M_0\cup M_1\cup M^{\prime}$ with $\partial M^{\prime}=-\partial M_0\sqcup \partial M_1$, $-M_0\cap M^{\prime}=-\partial M_0$ and $M_1\cap M^{\prime}=\partial M_1$.
			\item $F_{|{-M_0}}=f_0$ and $F_{|{M_1}}=f_1$.
			\item $F(M^{\prime})\subset A$.
\end{enumerate}
It has to be shown that for a given neighbourhood $U_\Delta$ of the diagonal $\Delta_X$ in $X\times X$ the oriented singular manifolds $(\widetilde{(M_0 \times M_0)}/\tau,\partial(\widetilde{(M_0 \times M_0)}/\tau);\sys{{f_0}}_{|\widetilde{{(M_0 \times M_0)}/\tau}})$ and $(\widetilde{(M_1 \times M_1)}/\tau,\partial(\widetilde{(M_1 \times M_1)}/\tau);\sys{{f_1}}_{|\widetilde{({M_1 \times M_1})/\tau}})$, chosen with respect to $U_\Delta$, are bordant in $\Omega_{2n}(\sys{X},pr(X\times A\cup A \times X \cup U_{\Delta}))$.\footnote{It was explained in \autoref{diagsec}, how these manifolds are chosen with respect to $U_\Delta$. \autoref{schlange}, which was introduced in the course of the construction, is used here.}  A large amount of work to prove this result has already been done in the proof of \autoref{halfun} because essentially the same construction works here. Namely, the manifold $\tilde{B}\subset W\times W$, which was chosen there as a fibre product of a certain Morse function, inherits an inverse image orientation satisfying the conditions needed to turn it into an oriented bordism between the oriented singular manifolds $\left(\widetilde{M_0 \times M_0},\partial(\widetilde{M_0 \times M_0});(f_0 \times f_0)_{|\widetilde{M_0 \times M_0}}\right)$ and $\left(\widetilde{M_1 \times M_1},\partial(\widetilde{M_1 \times M_1});(f_1 \times f_1)_{|\widetilde{M_1 \times M_1}}\right)$. Dividing these by $\tau$ as submanifolds of $W\times W$ gives a well-defined oriented bordism as desired.\\
In detail, there are only the following two facts to be proven since the rest follows then as in the proof of \autoref{halfun}. We use the same notation as in the cited proof.
\begin{itemize}
 \item[(a)] The manifold $\tilde{B}$ inherits an orientation which induces the orientations of $-(\widetilde{M_0\times M_0})$ and of $(\widetilde{M_1\times M_1})$ on its boundary.
\item[(b)] $\tilde{B}/\tau$ is a well-defined oriented manifold.
\end{itemize}
For (a), let $(x,y)\in \widetilde{M_0\times M_0}\cup\widetilde{M_1\times M_1}\subset\partial\tilde{B}.$ The boundary $\partial\tilde{B}$ inherits its orientation from an orientation of $\tilde{B}$. This was defined as an inverse image of a regular value of a certain function, so $\tilde{B}$ is oriented by an inverse image orientation depending on this function.\footnote{Consult Section 3.2 in \cite{MR0348781} to learn about the notion of inverse image orientation, called preimage orientation there.} Fortunately, the function in question was defined in a rather concrete and easy way in a neighbourhood of points in $\widetilde{M_0\times M_0}\cup\widetilde{M_1\times M_1}\subset\partial\tilde{B}.$ Namely, as constructed earlier in the proofs of \autoref{meinmorse} and \autoref{halfun},  $\tilde{B}=\bar{f}^{-1}(0)$, where $\bar{f}\colon W\times W\rightarrow\left[-1,1\right]$ is defined by  $(x,y)\mapsto f(x)-f(y)$  and $f\colon W\rightarrow [0,1]$  is defined in collaring-coordinates of $[0,1)\times\partial W\subset W$ as
  \[f(t,x)= \begin{cases}
         t & \text{ for } x\in M_0,\\	 
	 1-t & \text{ for } x\in M_1.	
         \end{cases}\]
For a point $x\in M_0\cup M_1$ the tangent space to $W$ can be written as $T_x W=N_x\oplus T_x(\partial W)$, where $N_x$ corresponds to the $t$-coordinate above. It follows that 
\[
d\bar{f}_{(x,y)}\colon T_{(x,y)}(W\times W)=T_x W\oplus T_y W=N_x\oplus T_x(\partial W)\oplus N_y\oplus T_y(\partial W)\rightarrow\mathbb R
\]can be computed as
\[
 (t_1,-,t_{n+2},-)\mapsto\begin{cases}
                      t_1-t_{n+2} & \text{ for } (x,y)\in \widetilde{M_0\times M_0}\\
		      (1-t_1)-(1-t_{n+2})=t_{n+2}-t_1 & \text{ for } (x,y)\in\widetilde{M_1\times M_1}.
                     \end{cases}
\]This difference in sign is crucial for the rest of the proof. Let $\left\{v_1^x,\ldots,v_n^x\right\}$ be a positive orientation of $M_0$ or $M_1$ depending on $x$ and let $n_x\in T_x W$ be an outward pointing vector, which here means that its $t$-coordinate is less than zero. Since by assumption $\partial W=-M_0\cup M_1\cup M^{\prime}$, it follows that $W$ is positively oriented by $\left\{n_x,-v_1^x,\ldots,v_n^x\right\}$ for $x\in M_0$ and by $\left\{n_x,v_1^x,\ldots,v_n^x\right\}$ for $x\in M_1$.  With an analogous notation for $y\in M_0\cup M_1$ we have that $W\times W$  is positively oriented by
\[
 \left\{(n_x,0),(-v_1^x,0),\ldots,(v_n^x,0),(0,n_y),(0,-v_1^y),\ldots,(0,v_n^y)\right\} \text{ for } (x,y)\in \widetilde{M_0\times M_0}
\]
\[
 \text{ or }  \left\{(n_x,0),(v_1^x,0),\ldots,(v_n^x,0),(0,n_y),(0,v_1^y),\ldots,(0,v_n^y)\right\}	\text{ for } (x,y)\in\widetilde{M_1\times M_1}.
\]
Replacing the vector $(0,n_y)$ by $(n_x,n_y)$ in the two bases does not change orientation and if we require $n_x$ and $n_y$ to have equal $t$-coordinates, then the resulting bases contain (not yet oriented) bases of $T_{(x,y)}\tilde{B}$, namely
\[
\left\{(-v_1^x,0),\ldots,(v_n^x,0),(n_x,n_y),(0,-v_1^y),\ldots,(0,v_n^y)\right\} \text{ for } (x,y)\in \widetilde{M_0\times M_0}
\] 
\[
 \text{ and }  \left\{(v_1^x,0),\ldots,(v_n^x,0),(n_x,n_y),(0,v_1^y),\ldots,(0,v_n^y)\right\}	\text{ for } (x,y)\in\widetilde{M_1\times M_1}.
\]
Viewing $(n_x,0)\in T_{(x,y)}(W\times W)$ gives an outward pointing vector $(n_x,0)$ with respect to $W\times W$ as well. For $(x,y)\in \widetilde{M_0\times M_0}$ the differential $d\bar{f}_{(x,y)}$ maps $(n_x,0)$ to  its $t$-coordinate whereas for $(x,y)\in \widetilde{M_1\times M_1}$ it is mapped to the negative of its $t$-coordinate. The image space of $d\bar{f}_{(x,y)}$, which is $\mathbb R$, is assumed to have standard orientation, so the inverse image orientation of the complement of $T_{(x,y)}\tilde{B}$ in $T_{(x,y)}(W\times W)$ is given by $(n_x,0)$ or $(-n_x,0)$ depending on which of them is mapped to a positive real number. Together with the fact that we know a positive orientation of $W\times W$ this gives positive orientations for $\tilde{B}$, namely
\[
\left\{(v_1^x,0),\ldots,(v_n^x,0),(n_x,n_y),(0,-v_1^y),\ldots,(0,v_n^y)\right\} \text{ for } (x,y)\in \widetilde{M_0\times M_0}
\] 
\[
 \text{ and }  \left\{(v_1^x,0),\ldots,(v_n^x,0),(n_x,n_y),(0,v_1^y),\ldots,(0,v_n^y)\right\}	\text{ for } (x,y)\in\widetilde{M_1\times M_1},
\]
since $(-n_x,0)$ is mapped to a positive real number by $d\bar{f}_{(x,y)}$ for $(x,y)\in \widetilde{M_0\times M_0}$ and $d\bar{f}_{(x,y)}(n_x,0)$ is positive for $(x,y)\in\widetilde{M_1\times M_1}$. The outward pointing vector $(n_x,n_y)$ can be swapped to the first position in the bases without changing orientation because $n$ is even. So on the boundary $\partial\tilde{B}$ the following orientations are induced:
\[
\left\{(v_1^x,0),\ldots,(v_n^x,0),(0,-v_1^y),\ldots,(0,v_n^y)\right\} \text{ for } (x,y)\in \widetilde{M_0\times M_0}
\] 
\[
 \text{ and }  \left\{(v_1^x,0),\ldots,(v_n^x,0),(0,v_1^y),\ldots,(0,v_n^y)\right\}	\text{ for } (x,y)\in\widetilde{M_1\times M_1},
\]which is exactly what needs to be shown for (a).
\\
For (b) it has to be shown that $\tau$ is orientation preserving on $\tilde{B}$. Let $(x,y)$ be an element of $\tilde{B}$ and suppose that the first row of the commutative diagram
\[
\begin{xy}\xymatrixcolsep{3.5pc}
 \xymatrix{
  0 \ar[r] & T_{(x,y)}\tilde{B}\ar[d]^{d\tau} \ar[r]^-{i} & T_{(x,y)}(W\times W) \ar[r]^-{d_{(x,y)}\bar{f}}\ar[d]^{d\tau} & \mathbb R\ar[r]\ar[d]^{\cdot(-1)} & 0 \\
  0 \ar[r] & T_{(y,x)}\tilde{B}\ar[r]^-{i} & T_{(y,x)}(W\times W)\ar[r]^-{d_{(y,x)}\bar{f}} & \mathbb R\ar[r] & 0
}
\end{xy}\label{1}
\]gives a positive orientation of $\tilde{B}$ at $(x,y)$. The orientation at the image under $\tau$ is then given by the second row of the preceding diagram. Since the dimension of $W$ is odd, we know that $\tau$ is orientation reversing on $W\times W$. From this together with the fact that the last vertical arrow in the diagram denotes the map that is multiplication with $(-1)$ it follows that the second row of the diagram induces a positive orientation on $\tilde{B}$ at $(y,x)$. But this means that $\tau$ is orientation preserving.

\end{proof}
As in the unoriented case, the symmetric squaring construction is natural in the oriented setting as well.
\begin{lem}[Naturality of the construction]\label{ornat}
Let $n\in \mathbb N$ be even, let $(X,A)$ and $(Y,B)$ be pairs of topological spaces and let $g\colon (Y,B)\rightarrow (X,A)$ be a continuous map. Then the following diagram is commutative
\[
\begin{xy}\xymatrixcolsep{3.5pc}
\xymatrix{
\Omega_n(Y,B) \ar[r]^{\sys{\cdot}} \ar[d]_{\Omega(g)} & \check{\Omega}_{2n}(\sys{(Y,B)})\ar[d]_{\check{\Omega}_{2n}(\sys{g})} \\
\Omega_n(X,A) \ar[r]_{\sys{\cdot}} & \check{\Omega}_{2n}(\sys{(X,A)})
}
\end{xy}
\]
\end{lem}
\begin{proof}
 The proof is similar to the proof of \autoref{unnat}.
\end{proof}

%% file: Thom_compat.tex
\section{Compatability via the orientation class}\label{sec:compat}

It was already mentioned in the first chapter that there exists a well-defined symmetric squaring map in \v{C}ech homology. In the preceding section we showed that there is a well-defined symmetric squaring map in \v{C}ech bordism as well. Furthermore, the well-known fundamental class transformation or fundamental class homomorphism\footnote{See for example \cite{MR548463} in Chapter I.6.} provides a transformation between bordism and homology. It is therefore natural to ask whether symmetric squaring is compatible with that transition. The answer is affirmative. Namely, in \autoref{compat} it will be shown that the diagram

\[
\begin{xy}\xymatrixcolsep{3.5pc}
\xymatrix{
\Omega_n(X,A) \ar[r]^{\mu} \ar[d]_{\sys{\cdot}} & H_n(X,A,\mathbb Z)\ar[d]_{\sys{\cdot}} \\
\check{\Omega}_{2n}(\sys{(X,A)}) \ar[r]_{\check{\mu}} & \check{H}_{2n}(\sys{(X,A)},\mathbb Z)
}
\end{xy}
\]
is commutative. Here $\mu$ is the fundamental class homomorphism and $\check{\mu}$ is induced by it in \v{C}ech homology. Note that the pair $(X,A)$ can be chosen to be a compact pair such that $X$ and $A$ are Euclidean Neighbourhood Retracts, which assures that \v{C}ech homology and bordism are isomorphic to singular homology and bordism respectively.\footnote{Compare \autoref{cechbord}.}

\subsection{Compatibility}
In \autoref{resultshom} we reviewed the most important properties of the homological symmetric squaring which will be used again in the following proof. The main ingredient in the argumentation will be that symmetric squaring maps fundamental classes to fundamental classes which was stated in \autoref{fundquadrat}.\\
The standard way of mapping from bordism to homology is the fundamental class homomorphism defined as follows.
\begin{defi}[fundamental class homomorphism]\label{fundhom}
  A passage from bordism to homology can be defined in the following way:
\begin{align*}
 \mu\colon\Omega_k(X,A)&\rightarrow H_k(X,A,\mathbb Z)\\
\left[M,\partial M; f\right]&\mapsto \mu(M,\partial M,f):=H_k(f)(\sigma_{\mathbf f}),
\end{align*}
where \begin{itemize}  \item $\sigma_{\mathbf f}\in H_k(M,\partial M,\mathbb Z)$ is the fundamental class and 
       \item $H_k(f)$ is the map which is induced by $f\colon (M,\partial M)\rightarrow (X,A)$ in homology.
      \end{itemize}
\end{defi}
The symmetric squaring maps in the two worlds of bordism and homology are compatible via the just defined fundamental class homomorphism.
\begin{prop}\label{compat}
Let $n\in\mathbb N$ be even and $(X,A)$ be a pair of topological spaces. Furthermore denote the fundamental class transformation by $\mu$. Then the diagram  
\[
\begin{xy}\xymatrixcolsep{3.5pc}
\xymatrix{
\Omega_n(X,A) \ar[r]^{\mu} \ar[d]_{\sys{\cdot}} & H_n(X,A,\mathbb Z)\ar[d]_{\sys{\cdot}} \\
\check{\Omega}_{2n}(\sys{(X,A)}) \ar[r]_{\check{\mu}} & \check{H}_{2n}(\sys{(X,A)},\mathbb Z)
}
\end{xy}
\]
is commutative.
\end{prop}

\begin{proof}
First of all it is necessary to take a look at how the fundamental class homomorphism $\mu$ induces a (well-defined) map $\check{\mu}$ between \v{C}ech-bordism and \v{C}ech-homology. Both groups in the bottom row of the above diagram are defined via an inverse limit, namely
\begin{align*}\begin{split}
&\check{\Omega}_{2n}\sys{(X,A)}\\
&= \varprojlim\left\{\Omega_{2n}(\sys{X},U_\Delta)|U_\Delta \supset pr(X\times A\cup A\times X\cup \Delta_X) \text{ open neighbourhood}\right\}\\
&\text{ and }\\
&\check{H}_{2n}(\sys{(X,A)},\mathbb Z)\\
&=  \varprojlim\left\{H_{2n}(\sys{X},U_\Delta,\mathbb Z)|U_\Delta \supset pr(X\times A\cup A\times X\cup \Delta_X) \text{ open neighbourhood}\right\}
\end{split}
\end{align*} by definition.\\
Thus $\check{\mu}$ is induced by defining levelwise \[\mu_{U_\Delta}\colon \Omega_{2n}(\sys{X}, U_\Delta)\rightarrow H_{2n}(\sys{X},U_\Delta;\mathbb Z)\] as the fundamental class homomorphism between the groups that occur in the limits as factors. This is a sensible way of defining a map because of the fact that for two different neighbourhoods $U_\Delta$ chosen as above, the inclusions in the limit induce maps that send the corresponding fundamental classes to one another.\footnote{This is a property of the inverse limit, compare \cite{MR1335915}.}\\
This leads to pursuing a levelwise strategy for the proof of the commutativity of the diagram as well. So choose an open neighbourhood $U_\Delta$ of  $pr(X\times A\cup A\times X\cup \Delta_X)$  in $\sys{X}$ and look at the diagram\[
\begin{xy}\xymatrixcolsep{1.9091pc}
 \xymatrix{
\Omega_n(X,A) \ar[r]^{\mu} \ar[d]_{\left(\sys{\cdot}\right)_{U_\Delta}} & H_n(X,A,\mathbb Z)\ar[d]_{\left(\sys{\cdot}\right)_{U_\Delta}} \\
\Omega_{2n}(\sys{X}, U_\Delta) \ar[r]_{\mu_{U_\Delta}} & H_{2n}(\sys{X},U_\Delta,\mathbb Z),
}
\end{xy}\]which is one level of the above diagram. The symmetric squaring was defined levelwise in bordism and homology and the bottom vertical arrow here shall denote these levelwise definitions in level $U_\Delta$. Then take an element $\left[B^n,\partial B^n,f\right]\in \Omega(X,A)$. It has to be shown that
\[
\sys{(\mu\left[B^n,\partial B^n,f\right])}=\mu_{U_\Delta}(\sys{\left[B^n,\partial B^n,f\right]}).
\] This can be done by computing
\begin{align*}
\sys{(\mu\left[B^n,\partial B^n,f\right])}
&=\sys{({f_{\ast}(\sigma)})} \text{, where }\sigma \text{ denotes the fundamental class of }(B^n,\partial B^n)\\
&=\sys{({f_{\ast}}(\left[\sum_{i=1}^{k}g_i \sigma_i\right]))} \text{, where the sum is }\sigma \text{ as a singular chain}\\ 
&=\sys{\left[\sum_{i=1}^k g_i f_{\sharp}(\sigma_i)\right]}  \\
&=\left[\sum\limits_{\genfrac{}{}{0pt}{}{i<j}{1\leq i,j\leq k}}g_i g_j pr_{\sharp}(f_{\sharp}(\sigma_i)\times f_{\sharp}(\sigma_j))\right]\\
&=\left[\sum\limits_{\genfrac{}{}{0pt}{}{i<j}{1\leq i,j\leq k}}g_i g_j pr_{\sharp}(f\times f)_{\sharp}(\sigma_i\times \sigma_j)\right]\text{, by a property of the }\\
&  \text{ simplicial cross product }\\
&=\left[\sum\limits_{\genfrac{}{}{0pt}{}{i<j}{1\leq i,j\leq k}} g_i g_j (\sys{f})_\sharp(pr^{B^n})_{\sharp}(\sigma_i\times\sigma_j)\right]\text{, by the definition of $\sys{f}$}\\
&=(\sys{f})_{\ast}\left[\sum\limits_{\genfrac{}{}{0pt}{}{i<j}{1\leq i,j\leq k}}g_i g_j pr^{B^n}_\sharp(\sigma_i\times\sigma_j)\right]\\
&=(\sys{f})_{\ast}\sys{\left[\sum_{i=1}^k g_i\sigma_i\right]}\\
&=(\sys{f})_\ast(\sys{\sigma})\\
&=\mu_{U_\Delta}\left[\widetilde{B^n\times B^n}/\tau,\partial(\widetilde{B^n\times B^n}/\tau),\sys{f}\right]\text{, because }\sys{\sigma} \text{ is the }\\
&\text{ fundamental class of } (\widetilde{B^n\times B^n}/\tau,\partial(\widetilde{B^n\times B^n}/\tau)) \\
&\text{ by \autoref{fundquadrat}.}\\
&=\mu_{U_\Delta}(\sys{\left[B^n,\partial B^n,f\right]})\\
\end{align*}
Passing to the limit proves the proposition.
\end{proof}

In unoriented bordism combined with homology with $\mathbb Z_2$-coefficients the same result holds without the restriction on the dimension. We can define the fundamental class homomorphism in a corresponding way and also the above proof works analogously.
\begin{rem}
Let $n\in\mathbb N$ be a natural number and $(X,A)$ be a pair of topological spaces. Furthermore denote the fundamental class transformation by $\mu$. Then the diagram
\[
\begin{xy}\xymatrixcolsep{3.5pc}
\xymatrix{
\mathcal{N}_n(X,A) \ar[r]^{\mu} \ar[d]_{\sys{\cdot}} & H_n(X,A,\mathbb Z_2)\ar[d]_{\sys{\cdot}} \\
\check{\mathcal{N}}_{2n}(\sys{(X,A)}) \ar[r]_{\check{\mu}} & \check{H}_{2n}(\sys{(X,A)},\mathbb Z_2)
}
\end{xy}
\]
is commutative.
\end{rem}

%% file: computations.tex
\chapter{Computations}

As soon as the attempt of computing the symmetric squaring map in homology or bordism for specific topological spaces is made, difficulties arise due to the fact that the coordinate switching involution does not induce a free action on a squared space. There always is the diagonal as a fixed point set.\\
One tool for solving problems with fixed point sets is called the Borel construction. For a compact Lie group $G$ and a $G$-space\footnote{For topological groups $G$, the notion of a $G$-space $X$ is defined to be a topological space $X$ together with a $G$-action on $X$. The theory of topological G-spaces is explained in more detail in \cite{MR0413144}.} $X$, it uses the universal principal bundle $G\rightarrow EG\rightarrow BG$ to produce a space strongly related to $X$, which has a free action associated to it.\\
In the first section of this chapter, we shortly review some facts about the Borel construction. These are used in the second section together with the theory of $G$-spaces to prove a theorem containing an alternative description of the homology group of a symmetric squared space that is hopefully easier to compute.

\section{The Borel construction}
The so-called Borel construction is explained in more detail in \cite{MR1236839}. It was introduced by A. Borel in \cite{MR0116341} in order to study the cohomology of $G$-spaces.

\begin{defi}[Borel construction]\label{def:borel}
Let $G$ be a compact Lie group. Furthermore, let $X$ be a $G$-space and $EG$ the universal free $G$-space. Define the space $X_G$ to be the orbit space of the diagonal action on the product $EG\times X$, i.e. $G$ acts on $EG\times X$ via $g(x,y)=(gx,gy)$. The space $X_G$ then is the total space of the bundle $X\rightarrow X_G\rightarrow BG$ associated to the universal principal bundle $G\rightarrow EG\rightarrow BG$, where $BG:=EG/G$ is the classifying space of $G$.
\end{defi}
What is done here is that to a $G$-action on $X$ that might have fixed points, there is associated a free $G$-space which is topologically very similar to $X$ since the universal free $G$-space $EG$ is contractible disregarding the $G$-action.

\begin{example}\label{ex:sybo}
In the case of the symmetric squaring construction, we deal with $\mathbb Z_2$-spaces. The universal principal bundle here is\footnote{Compare Example 1B.3 in \cite{MR1867354}.} $\mathbb Z_2\rightarrow S^\infty\rightarrow \mathbb{RP}^\infty$. For any topological space $X$ the $\mathbb Z_2$-action induced by the involution $(x,y)\mapsto(y,x)$ on $X\times X$ has the diagonal as a fixed point set. Using the Borel construction, we associate to this the space $X\times X\times S^\infty$ with the free $\mathbb Z_2$-action induced by $(x,y,z)\mapsto(y,x,-z)$ and $(X\times X)_{\mathbb Z_2}$ is defined to be its orbit space.
\end{example}
In the next section we will prove a result relating the homology groups of $G$-spaces to the homology groups of the Borel construction.
\begin{rem}
 \autoref{def:borel} is commonly used to define an equivariant cohomology theory $H^\ast_G(-)$ via $H^\ast_G(X):=H^\ast(X_G)$.
\end{rem}

\section{Using G-spaces}\label{sec:G-spaces}
 
Our purpose is to find a relation between the homology groups of symmetric squared spaces and homology groups that are possibly easier to compute. One candidate of such is the Borel construction as we will see below. Before we can prove the main result of this section, we have to cite a famous (co)homological mapping theorem.\\
The Vietoris-Begle Mapping Theorem was originally proven in \cite{MR0035015,MR0082097}. In \autoref{borel-hom} below the following homology version of this theorem is needed, which can be found in \cite{MR753049}. As before $\check{H}_\ast$ denotes \v{C}ech homology, furthermore $\tilde{H}_\ast$ denotes reduced homology. 
\begin{thm}[Vietoris-Begle Mapping Theorem]\label{vie-beg}
 Let $f\colon X\rightarrow Y$ be a continuous surjective mapping of metrizable compacta, for which
\[
\tilde{H}_i(f^{-1}(y);G)=0 \text{ for all } y\in Y \text{ for } i\leq n. 
\]Then if G is a countable group, the induced homomorphism \[\check{H}_q(f)\colon \check{H}_q(X;G)\rightarrow \check{H}_q(Y;G)\] is an isomorphism for $0\leq q\leq n$ and an epimorphism for $q=n+1$.
\end{thm}
One technical difficulty we have to overcome in the following is that we want to use the preceding theorem for the non-compact space $EG$. As in \cite{MR0413144}, the way out here is looking at approximations of the Borel construction rather than at the construction itself.
\begin{rem}[Approximation of the Borel construction]\label{borel-approx}
As before, let $G$ be a compact Lie group. Furthermore, take a universal principal $G$-bundle $EG\rightarrow BG$, whose classifying space $BG$ is a CW-complex with finite $N$-skeleton $BG^N$ for all $N$. Now, let $EG^N$ be the inverse image of $BG^N$ under the above projection map. This construction is done in a way such that $EG^N$ is compact and\[
X\times_G EG^N \hookrightarrow X_G 
\]is $N$-connected for all $N$.\\
This approximation makes it possible to reduce later statements to the $N$th approximation of the Borel construction..
\end{rem}
In cohomology the proposition given next holds in greater generality and is stated and proven as Proposition VII.1.1 in \cite{MR0413144}. The proof of the homology case given here is based on that proof.
\begin{note}Note that as it is common in the theory of $G$-spaces, given a $G$-space $X$, the fixed point set of the $G$-action on $X$ is denoted by $X^G$  and the orbit space is denote by $X/G$ in the sequel. Recall furthermore that $X_G$ was defined as the Borel construction in the preceding section. 
\end{note}
With this notation in mind, we can finally prove the announced homology theorem for symmetric squared spaces.
\begin{prop}\label{borel-hom}
 Let $G=\mathbb Z_p$ for $p$ prime and let $X$ be a compact and metrizable $G$-space. Then the projection
\[
\phi\colon X_G\rightarrow X/G
\] induces an isomorphism 
\[
\phi_\ast\colon \check{H}_i(X_{G},X^{G}\times BG)\rightarrow \check{H}_i(X/G,X^{G})
\]for coefficients in a countable group.
\end{prop}
\begin{proof}
Using \autoref{borel-approx}, we only prove the statement for the $N$th approximation for all $N$.\\
 To be able to use the Vietoris-Begle Mapping Theorem as stated above, examine $\phi^{-1}(\left[x\right])$ for $\left[x\right]\in X/G$. From Proposition I.4.1 in \cite{MR0413144} it follows that
\[
\phi^{-1}(\left[x\right])\approx G(x)\times_{G} EG^N \approx G/G_x\times_{G}EG^N.
\]Together with the computation
\[
(G/H)\times_{G} EG^N\approx (G\times_G EG^N)/H\approx EG^N/H\approx BH^N \text{ for subgroups } H\subset G
\] this gives $\phi^{-1}(\left[x\right])\approx BG_x^N$. So the requirements for the use of the Vietoris-Begle Mapping Theorem are satisfied for the restriction map $\tilde{\phi}\colon X_G^N\setminus (X^G\times BG^N)\rightarrow (X\setminus X^G)/G$, namely
\[
\tilde{H}_i\left(\phi^{-1}\left(\left[x\right]\right)\right)=0 \text{ for all } x\notin X^G \text{ and for all } i\in \mathbb Z
\] with coefficients in countable groups. Applying the Vietoris-Begle mapping \autoref{vie-beg} proves the proposition.
\end{proof}
We can use \autoref{ex:sybo} to formulate the preceding theorem in the special case of symmetric squaring.
\begin{example}[Symmetric squaring]
Let X be a metrizable compact space. The projection map \[
\phi\colon(X\times X)_{\mathbb Z_2}\cong \left(\left((X\times X\times S^\infty)\right)/{\mathbb{Z}_2}\right)\rightarrow \sys{X}
\]
induces an isomorphism\[
\phi_\ast\colon\check{H}_i\left(\left(X\times X\right)_{\mathbb Z_2}, \Delta\times \mathbb{RP}^\infty\right)\rightarrow \check{H}_i\left(\sys{X},\Delta\right)
\] in \v{C}ech homology, where $\Delta$ denotes the diagonal in $X\times X$.
\end{example}
This gives an alternative way of computing homology groups of symmetric squared spaces $\sys{X}$ involving a free action instead of one with a nonempty fixed point set.

%% file: outlook.tex
\chapter{Perspectives}

The main achievement of this thesis is the transport of the symmetric squaring construction to bordism and relating the newly achieved bordism symmetric squaring to the homological symmetric squaring which was known before. The idea of transporting symmetric squaring from homology to bordism did not arrive completely out of the blue. There is a strong relationship between homology and bordism, demonstrated for example by the existence of a canonical map from bordism to homology\footnote{The fundamental class map explained in \autoref{fundhom}.} or by the fact that bordism groups can be computed via homology groups in many cases\footnote{Compare Theorems 8.3 and 12.9 in \cite{MR548463}.}.\\
Therefore, promising future research goals in connection with symmetric squaring can be set in areas related to bordism or to other generalised homology theories which are defined in a similar way to singular bordism.\\
One example which fits into the latter category is the geometric definition of $K$-homology. It was introduced by Baum and Douglas in \cite{MR679698} and proven to be equivalent to the analytic definition of $K$-homology in \cite{MR2330153}. It is similar to singular bordism in the sense that it deals with $K$-cycles which are defined to be triples $(M,E,\phi)$, where $\phi\colon M\rightarrow X$ is a continuous map from a smooth compact manifold $M$ to a space $X$ just as in the definition of singular manifolds. A $K$-cycle, however, has two additional structures to it. First of all, the smooth compact manifold $M$ is assumed to be equipped with a $Spin^c$-structure and secondly a smooth Hermitian vector bundle $E$ on $M$ is taken into account.\\
Just as with singular bordism, there is introduced a bordism relation on $K$-cycles and the set of equivalence classes with respect to this relation form homology groups $K(X)$ for the space $X$. As there is more structure to the objects that this theory is dealing with, there is also more structure to the bordism relation introduced in this context. Namely two $K$-cycles are considered to be isomorphic if there exist isomorphisms compatible with the different structures such as vector bundles and spinor bundles and a bordism of $K$-cycles has to feature the analogous additional $Spin^c$- and vector bundle structures as $K$-cycles do. What remains the same in comparison with singular bordism is that two $K$-cycles $(M_0,E_0,\phi_0)$ and $(M_1,E_1,\phi_1)$ are bordant if there exists a bordism between them which has an underlying manifold containing the disjoint union of $M_0$ and $-M_1$ as a regular submanifold in its boundary and which respects all structures as noted above.\\
Although in the resulting groups disjoint union is the group operation as it is in singular bordism as well, there is still one more difference between $K$-homology groups and singular bordism groups. Since the objects dealt with are $Spin^c$-manifolds, they are equipped with $Spin^c$-vector bundles and this gives rise to the so-called bundle modification which gives rise to another extra relation in $K$-theory as compared to singular bordism.\footnote{The last three paragraphs are based on the exposition of geometric $K$-homology in \cite{MR2330153}.} \\
Aiming at transporting the symmetric squaring construction to geometric $K$-homology, it should be possible to use the better part of the achievements in this thesis for all situations where $K$-homology is similar to singular bordism. It remains to show, however, that the symmetric squaring is as well compatible with the additional structures of $K$-cycles. Namely, it would be desirable to know how to extend symmetric squaring of manifolds to vector bundles or $Spin^c$-structures on the manifolds in a way that the results are bordant $K$-cycles as soon as the manifolds are underlying manifolds of bordant $K$-cycles in the first place. In this context it would also be necessary to shed a light on the behaviour of symmetric squaring with respect to bundle modification and direct sums of vector bundles.

Another attempt to generalise symmetric squaring further could lie in detecting superior structures in the proofs of this thesis that either indicate possibilities of further generalisations or provide obstructions to such. In connection with smooth manifolds, common structures to think of are structure groups of frame bundles of manifolds\footnote{Compare Section 4.4 in \cite{MR1841974} for a definition of frame bundles.} or accordingly $G$-structures on smooth manifolds. For some Lie subgroup $G$ of $GL(n,\mathbb R)$, a $G$-structure on an $n$-manifold is defined to be a reduction of the structure group of the frame bundle of the manifold to $G$.\footnote{We refer to the first chapter of \cite{MR1336823} for details about $G$-structures.} \\
So far in this thesis, we dealt with unoriented and oriented compact smooth manifolds. An oriented compact smooth $n$-manifold, for example, has an $SO(n)$-structure.\footnote{Using a Riemannian structure on compact smooth $n$-manifolds gives an $O(n)$-structure which can be reduced further to an $SO(n)$-structure via orientability. Compare the examples of $G$-structures in I.2 of \cite{MR1336823}.} First of all it is natural to try to think about how this given $SO(n)$-structure on an oriented $n$-manifold $M$ induces an analogous structure on the symmetric square $\sys{M}$ and even more to find out whether this induced structure can be reduced further. In \autoref{halfor}, we could prove that a lifting of symmetric squaring to oriented bordism is well-defined for even dimensional manifolds only. It is of interest to know whether any specific properties of the $SO(n)$-structure or its induced structure on the symmetric squared manifold directly imply either that symmetric squaring can be transported to oriented bordism in even dimensions or that a generalisation to oriented bordism is not possible in odd dimensions. As soon as such obstructions or implications are identified, it is worth attempting to prove further generalisations of the symmetric squaring construction to other bordism theories in which manifolds with $G$-structures are involved having the properties corresponding to those of $SO(n)$-structures in even dimensions.\\   
To put it all in a nutshell, one perspective of generalising symmetric squaring further lies in answering the following question: What (minimal) properties have to be satisfied by a $G$-structure on a manifold $M$ or an induced structure on $\sys{M}$ respectively in order to allow symmetric squaring to map bordant manifolds to bordant manifolds in a way compatible with these group structures introduced on each manifold?